                   \def\version{6 January, 2021}		       %
\documentclass[reqno,12pt]{amsart} 

\usepackage[colorlinks=true,linkcolor=blue,citecolor=blue]{hyperref}
\usepackage[dvipsnames]{xcolor}

\usepackage{upgreek }

\usepackage[utf8]{inputenc}
\usepackage{amsmath} 
\usepackage[mathscr]{eucal}
\usepackage{a4wide}
\usepackage{amssymb,amsthm,amsmath}
\usepackage{mathrsfs} 
\usepackage{dsfont}   
\usepackage{xspace}
\usepackage{srcltx} 
\usepackage{xspace}
\usepackage{url}      
\usepackage{mathtools}
\usepackage{color}    
\usepackage{tikz}     
\usepackage{enumerate} 

\usepackage{centernot} 

\numberwithin{equation}{section}

\def\1{{\mathchoice {1\mskip-4mu\mathrm l}      
{1\mskip-4mu\mathrm l} 
{1\mskip-4.5mu\mathrm l} {1\mskip-5mu\mathrm l}}} 
\renewcommand{\d}{{\rm d}}

\newcommand{\R}     {\mathbb{R}} 
 
\newcommand{\N}     {\mathbb{N}} 
\renewcommand{\P}   {\mathbb{P}} 
 
\newcommand{\E}     {\mathbb{E}}

\newcommand{\smfrac}[2]{{\textstyle{\frac {#1}{#2}}}}

\newcommand{{\Mi}}{{\rm Mi}}
\newcommand{{\Ma}}{{\rm Ma}}
\newcommand{{\Me}}{{\rm Me}}

\newcommand{\ssup}[1] {{{\scriptscriptstyle{({#1}})}}} 
 
\newtheorem{theorem}{Theorem}[section] 
\newtheorem{lemma}[theorem]{Lemma} 
\newtheorem{prop}[theorem] {Proposition} 
\newtheorem{cor}[theorem]  {Corollary}

\theoremstyle{definition}

\newcommand{\eps}{\varepsilon} 

\newcommand{\Gcal}   {{\mathcal G }} 
 
\newcommand{\Ical}   {{\mathcal I }} 
\newcommand{\Jcal}   {{\mathcal J }}

\newcommand{\Mcal}   {{\mathcal M }} 
\newcommand{\Ncal}   {{\mathcal N }} 
 
\newcommand{\Pcal}   {{\mathcal P }} 
 
\newcommand{\Rcal}   {{\mathcal R }}

\newcommand{\e}   {{\operatorname e }}

\definecolor{Red}{rgb}{1,0,0}

\definecolor{darkspringgreen}{rgb}{0.09, 0.45, 0.27}

\begin{document}
\title[LDP for the cluster sizes of a sparse Erd\H{o}s-R\'enyi graph]
{A large-deviations principle\\
\medskip for all the cluster sizes\\
\medskip of a sparse Erd\H{o}s-R\'enyi graph}

\author[L. Andreis]{Luisa Andreis}
\address{WIAS, Mohrenstra{\ss}e 39, 10117 Berlin}
\email{luisa.andreis@wias-berlin.de}
\author[W. König]{Wolfgang König}
\address{TU Berlin and WIAS, Mohrenstra{\ss}e 39, 10117 Berlin}
\email{wolfgang.koenig@wias-berlin.de}
\author[R.I.A. Patterson]{Robert I. A. Patterson}
\address{WIAS, Mohrenstra{\ss}e 39, 10117 Berlin}
\email{robert.patterson@wias-berlin.de}

\maketitle
\centerline{\small(\version)} 
\begin{abstract} 
Let $\Gcal(N,\frac 1Nt_N)$ be the Erd\H{o}s-R\'enyi graph with connection probability $\frac 1Nt_N\sim t/N$ as $N\to\infty$ for a fixed $t\in(0,\infty)$.
We derive a large-deviations principle  for the empirical measure of the sizes of all the connected components of $\Gcal(N,\frac 1Nt_N)$, registered according to  microscopic sizes (i.e., of finite order), macroscopic ones (i.e., of order $N$), and mesoscopic ones (everything in between).
The rate function  explicitly describes the microscopic and macroscopic components and the fraction of vertices in components of  mesoscopic sizes.
Moreover, it clearly captures the well known phase transition  at $t=1$ as part of a comprehensive picture. The proofs rely on elementary combinatorics and on known estimates and asymptotics for the probability that subgraphs are connected. We also draw conclusions for the strongly related model of the multiplicative coalescent, the Marcus--Lushnikov coagulation model with monodisperse initial condition, and its gelation phase transition.
\end{abstract}





\setcounter{section}{0} 

\section{Introduction}\label{Intro}

In this paper, we study the Erd\H{o}s-R\'enyi random graph $\Gcal(N,\frac{1}Nt_N)$, that is, the random graph on the vertex set $[N]=\{1,\dots,N\}$, where each two distinct vertices are independently connected with probability $\frac{1}Nt_N$. We will be working in the sparse regime, i.e., we assume that $\lim_{N\to \infty}t_N=t$ for some fixed $t\in(0,\infty)$. This is the regime in which the famous phase transition of the emergence of a giant cluster at $t=1$ occurs, which was detected and characterised for the first time in the seminal paper \cite{Erd60}. For an extensive overview on the model see the classical reference \cite{Bol01}.

Our new contribution  in this paper is a comprehensive study of the family of the sizes of all the connected components of  $\Gcal(N,\frac{1}Nt_N)$, registered according to the  asymptotic order of the size in the limit as $N\to\infty$. We distinguish here {\em microscopic components} (i.e., with size of order one), {\em macroscopic components} (i.e., size of order $N$, usually referred to as {\em giant components}) and {\em mesoscopic} ones (everything in between). 
We  summarize all this information in terms of two empirical measures and derive a large-deviation principle (LDP) for them. Our rate function is rather explicit.

Such a principle gives information about the exponential decay rate of all sorts of events, e.g, the emergence of more than one giant cluster or the presence of a non-trivial proportion of vertices in mesoscopic components. Moreover, the minimizers of the rate function represent the most likely configurations of the graph, which is expressed in terms of a law of large numbers  for the objects that satisfy the LDP. In this way, we recover the mentioned phase transition and collect detailed information about the statistics of the sizes of all the components, both in the subcritical regime (where no giant component occurs) and the supercritical one.

Many investigations of the Erd\H{o}s-R\'enyi graph and other random graphs in the sparse regime rely on approximations of subgraphs with certain Galton--Watson trees and other branching processes. We would like to stress that our approach does not use such arguments and is therefore an alternate ansatz.

In Section \ref{sec-Micromacro}, we introduce our approach, in Section \ref{sec-Results}, we formulate our main results about large deviations and in Section \ref{sec-phasetrans} their consequences for the phase transition,  and in Section \ref{subsect:ERRG}  we give a literature survey.

Our original interest in this study was triggered by a desire to understand {\em random particle processes with coagulation}, in particular its simplest variant, the {\em Marcus--Lushnikov model} with multiplicative coagulation kernel. We introduce this process and its connections with our work on the LDP for the Erd\H{o}s-R\'enyi graph in Section \ref{sec-Coagulation}.

Another highly interesting connection appears with a LDP-proof for the well-known {\em Bose--Einstein condensation} phase transition that appears in the free (i.e., non-interacting) Bose gas; we will explain the similarities and the differences in Section \ref{sec-BEC}.

\subsection{Micro- and macroscopic empirical measures}\label{sec-Micromacro}

Let us introduce the main objects that we study in this paper.
For the remainder of the paper, we fix $t\in(0,\infty)$ and  will be working with the graph $\Gcal(N,\frac{1}Nt_N)$, where $t_N=t+o(1)$ as $N\to\infty$. 
By 
\begin{equation}\label{componentsizes}
S_1^{\ssup N}\geq S_2^{\ssup N}\geq \dots\geq S_{n}^{\ssup N}\geq 1,\qquad\sum_{i=1}^{n}S_i^{\ssup N}=N,
\end{equation}
we denote the sizes of all the connected components of $\Gcal(N,\frac 1Nt_N)$, ordered in a decreasing way ($n\in\{1,\dots,N\}$ is the number of components). We want to describe the entire family $(S_i^{\ssup N})_{i\in\{1,\dots,n\}}$ in the limit $N\to\infty$. This is a comprehensive object, which contains several scales. In oder to adequately describe the most important two scales, it will be convenient to work with two {\em empirical measures} of the component sizes in the microscopic and macroscopic size ranges:
\begin{equation}\label{empiricalmeasures}
{\rm Mi}^{\ssup N}=\frac 1N\sum_{i=1}^{n}\delta_{S_i^{\ssup N}}\qquad \mbox{and}\qquad {\rm Ma}^{\ssup N}=\sum_{i=1}^{n}\delta_{\frac 1N S_i^{\ssup N}}.
\end{equation}
Intuitively, while ${\rm Mi}^{\ssup N}$ registers the proportion of components of \lq microscopic\rq\ sizes $1,2,3,\dots$ on the scale $N$, ${\rm Ma}^{\ssup N}$ registers the components  of \lq macroscopic\rq\ sizes, i.e.\ order $N$. Note that each of the two measures admits a one-to-one map onto the vector $(S_i^{\ssup N})_{i\in\{1,\dots,n\}}$ for fixed $N\in\N$ and therefore contains all the information  contained in the vector. However, in the limit $N\to\infty$, they will be able to describe only the statistics of the microscopic, respectively macroscopic, part of the particle configuration. We would like to stress here that this issue lies at the heart of the phase transition of the emergence of a giant component, i.e., a macroscopic size. 

Here is a non-technical, intuitive explanation: in the limit as $N\to \infty$, all sizes $S_i^{\ssup N}$ that somehow diverge, will vanish from the support of $\Mi^{\ssup N}$ \lq at infinity\rq, and all sizes $S_i^{\ssup N}$ that are $\ll N$ will vanish from the support of $\Ma^{\ssup N}$ \lq at zero\rq. Hence, $\Mi^{\ssup N}$ may leak out mass at infinity, and $\Ma^{\ssup N}$ at zero. It is by no means automatic that all the mass that leaks out from the microscopic part at infinity enters the macroscopic part at zero. In order to control that, we also need to take care of the {\it mesoscopic} mass, coming from particle masses $1\ll S_i^{\ssup N}\ll N$. Since here a lot of scales are contained (indeed, a continuum of scales), we will not be able to say anything about these sizes, but only about the total proportion of vertices belonging to such components. 

The famous phase transition (proved first in \cite{Erd60}) says that, for $t\leq 1$, there is no loss of mass from $\Mi^{\ssup N}$ (i.e., the first moment stays equal to one in the limit), and $\Ma^{\ssup N}$ convergesd to the zero measure (i.e., loses all its mass), while for $t>1$, the total mass of $\Mi^{\ssup N}$ loses a positive amount equal to that retained by $\Ma^{\ssup N}$ in the limit, and this results in a single Dirac measure. In both cases, the mesoscopic part vanishes,  even though mesoscopic components are present in the graph with high probability, but their proportion is negligible. 

These are assertions of the type of laws of large numbers. However, in the setting of a large-deviation principle as we are working here, we will obtain significant results also about the probabilities of several very unlikely events, like the emergence of non-trivial mesoscopic total mass, of more than one giant cluster and of different statistics of microscopic component sizes. Note that we decided to work on probabilities that are on an exponential scale $N$ and for connection probabilities of the form $\frac 1N(t+o(1))$ for general $t\in(0,\infty)$. This excludes for example all (highly interesting) phenomena that occur with respect to cluster sizes of order $N^{2/3}$ when considering more specified connection probabilities of the size $\frac 1N(1+ c N^{-1/3})$; see \cite{Ald97}.

Now let us give a more technical explanation of the issue about possible losses of masses, which will also set the frame for the mathematical treatment. We will conceive the discrete measure ${\rm Mi}^{\ssup N}=(\Mi^{\ssup N}_k)_{k\in\N}$ as a random element of the sequence set $\Ncal=\bigcup_{c\in[0,1]}\Ncal(c)$, where
\begin{equation}
\Ncal(c)=\Big\{\Lambda=(\lambda_k)_{k\in\N}\in[0,\infty)^\N\colon \sum_{k\in\N}k \lambda_k=c\Big\},\qquad c>0.
\end{equation}
We equip $\Ncal=\{\Lambda\colon\sum_k k\lambda_k\leq 1\}$ with the topology of coordinate-wise convergence, which makes it compact by the Bolzano--Weierstrass theorem combined with Fatou's lemma.

The point measure ${\rm Ma}^{\ssup N}$ is a random element of the set $\Mcal:=\bigcup_{c\in[0,1]}\Mcal_{\N_0}((0,1];c)$, where 
\begin{equation}
\Mcal_{\N_0}((0,1];c)=\Big\{\alpha\in \Mcal_{\N_0}((0,1])\colon \int_{(0,1]}x\,\alpha(\d x)=c\Big\},
\end{equation}
and $\Mcal_{\N_0}((0,1])$ is the set of all measures on $(0,1]$ with values in $\N_0=\left\{0\right\}\cup \N$. We equip $\Mcal$ with the topology that is induced by functionals of the form $\mu \mapsto \int_{(0,1]}f(x)\,\mu(\d x)$ where $f\colon(0,1]\to\R$ is continuous and compactly supported. We sometimes write the elements of $\Mcal$ as $\alpha=\sum_j\delta_{\alpha_j}$ with $1\geq \alpha_1\geq\alpha_2\geq \dots>0$ and $\sum_j \alpha_j\leq 1$, where $j$ extends over a finite subset of $\N$ or over $\N$. Then convergence is equivalent with the pointwise convergence of each of the atoms. By similar arguments as for $\Ncal$, also $\Mcal$ is compact. We equip the product of $\Ncal$ and $\Mcal$ with the product topology, so that it is also compact.

Important quantities are the expectations of the sub-probability distributions $\Lambda\in\Ncal$ respectively $\alpha\in\Mcal$, i.e., the maps
$$
\Lambda \mapsto c_\Lambda:=\sum_{k\in\N}k \lambda_k\qquad
\text{and}\qquad
\alpha \mapsto c_\alpha:=\int_{(0,1]}x\,\alpha(\d x).
$$
Note that they are not continuous in the respective topologies, but only  lower semicontinuous, according to Fatou's lemma. Indeed, even though the microscopic and macroscopic expectations $c_{\Mi^{\ssup N}}=\sum_k k{\rm Mi}^{\ssup N}_k$ and $c_{\Ma^{\ssup N}}=\int_{(0,1]} x\,{\rm Ma}^{\ssup N}(\d x)$ are each equal to one for any $N$, they may (and will) lose mass in the limit $N\to\infty$. We sometimes call $c_\Lambda$ and $c_\alpha$ the {\em total masses} of the microscopic, respectively macroscopic, configuration $\Lambda$ and $\alpha$, since they stand for the total number of particles, after scaling.

The mathematical treatment of the mesoscopic part of the component sizes is more technical, as it requires the introduction of two cutting parameters $R\in\N$ and $\eps\in(0,1)$. Indeed, a size $S_i^{\ssup N}$ is called $(R,\eps)${\em -mesoscopic} if $R<S_i^{\ssup N}<\eps N$, and the definition of mesoscopic sizes requires making the limit $N\to\infty$, followed by $R\to\infty$ and $\eps\downarrow 0$. There are several scales (indeed, a continuum of scales) contained in this part and in this regime it does not seem reasonable to consider an empirical measure for this part; therefore we will consider only the total proportion of mesoscopic  vertices.

Let us remark that our choice of considering exclusively the size of each component, disregarding its bond structure, comes from the interest in coagulation processes, where only the sizes matter; see Section~\ref{sec-Coagulation}. An extension of our work to empirical measures of the components seems to require only moderate additional work, at least as it concerns the microscopic part. See Section \ref{subsect:ERRG} for earlier LDP-investigation of the components as subgraphs.


\subsection{Our results: large-deviations principles} \label{sec-Results}

In this section, we present all our results on the LDP satisfied by the empirical measure of statistics of  component sizes of the Erd\H {o}s--R\'enyi graph  $\Gcal(N,\frac{1}Nt_N)$, the random graph on $[N]=\{1,\dots,N\}$ with connection probability $\frac{1}Nt_N$, and we assume that $t_N=t+o(1)$ with fixed $t\in(0,\infty)$.  In Section \ref{sec-phasetrans} we will draw conclusions about the phase transition from that. 
Our main result is the following description of the two empirical measures $\Mi^{\ssup N}$ and $\Ma^{\ssup N}$ in terms of a joint large-deviations principle (LDP). 

\begin{theorem}[LDP for the empirical measures]\label{thm-LDP}
As $N\to\infty$, the pair $({\rm Mi}^{\ssup N}, {\rm Ma}^{\ssup N})$ satisfies a large-deviations principle with speed $N$ and rate function 
$$
(\Lambda,\alpha)\mapsto I(\Lambda,\alpha;t)
=\begin{cases}
I_{\rm Mi}(\Lambda;t) + I_{\rm Ma}(\alpha;t)+(1-c_{\Lambda}-c_{\alpha})\Big(\frac t2-\log t\Big),&\mbox{if }c_\Lambda+c_\alpha\leq  1,\\
\infty&\mbox{otherwise,}
\end{cases}
$$
where we write $\Lambda=(\lambda_k)_{k\in\N}$ and 
\begin{eqnarray}
I_{\rm Mi}(\Lambda;t)&=& 
 \sum_{k=1}^{\infty} \lambda_k
    \log\frac{k!t\lambda_k}{\e\, k^{k-2}}+c_\Lambda\Big(1+ \frac t2-\log t \Big),\qquad c_\Lambda= \sum_{k=1}^{\infty}k \lambda_k,\label{I_mi}\\
I_{\rm Ma}(\alpha;t)&=& \int_0^1 \Big[x\log\frac{x}{1-\e^{-tx}}+\frac t2 x(1-x)\Big]\,\alpha(\d x),\qquad c_\alpha=\int_{(0,1]} x\,\alpha(\d x).\label{Stdef}
\end{eqnarray}
\end{theorem}

The proof of this theorem is in Section \ref{sec-ProofLDP}; it is based on an explicit combinatorial formula for the joint distribution of all the component sizes, followed by analysis of the arising exponential rates. We organised the three terms of $I$ in the way in which they were derived from the influences of the three parts (micro, macro and meso) in the course of the proof, even though this leads to a cancellation of terms involving $c_\Lambda$ and $c_\alpha$. This implies also that separate conclusions about the microscopic and the macroscopic parts can conveniently be made (see Corollaries \ref{cor-LDP_mi} and \ref{cor-LDP_ma}). Informally, in \eqref{I_mi} the terms involving $\lambda_k$, $\e$ and $k!$ in the logarithm derive from the combinatorial number of possibilities to decompose $[N]$ into the requested configuration of subsets, the term $k^{k-2}$ and the $t$ in the logarithm stem from the probability that these subsets are connected, and all the other terms from the probability that any of these subsets is not connected with the remainder. This interpretation is not immediate, since a number of asymptotic manipulations have been made during the proof. Similar remarks apply to \eqref{Stdef}. Interestingly, on the right-hand side of \eqref{I_mi} we see, up to normalization, a relative entropy of $(k\lambda_k)_{k\in\N}$ with respect to the {\em Borel distribution} ${\rm Bo}_\mu(k)=\e^{-\mu k}(\mu k)^{k-1}/k!$ for a particular choice of $\mu$; a fact that will be crucial in the analysis of minimizers of the rate function, see the proofs of Corollaries \ref{cor-LDP_mi} and \ref{cor-LDP_ma} and of Theorem \ref{thm-phasetrans}.

Let us recall the notion of an LDP: Theorem~\ref{thm-LDP} says that, for any open set $G\subset \Ncal\times\Mcal $ respectively closed set $F\subset \Ncal\times\Mcal$,
\begin{eqnarray*}
\liminf_{N\to\infty}\frac 1N\log\P_N(({\rm Mi}^{\ssup N}, {\rm Ma}^{\ssup N})\in G)&\geq& -\inf_G I(\cdot;t),\\ \limsup_{N\to\infty}\frac 1N\log\P_N(({\rm Mi}^{\ssup N}, {\rm Ma}^{\ssup N})\in F)&\leq& -\inf_F I(\cdot;t),
\end{eqnarray*}
where we wrote $\P_N$ for the probability measure for $\Gcal(N,\frac 1Nt_N)$. 
For a comprehensive presentation of the theory of large-deviations, see e.g.\ \cite{DZ10}. It is not difficult to see that since the rate function $I(\cdot,\cdot;t)$ is lower semicontinuous and $\Ncal\times\Mcal $ is compact, it is even a good rate function, i.e., its level sets $\{(\Lambda,\alpha)\colon I(\Lambda,\alpha;t)\leq r\}$ are compact for any $r$.

It is well-known in the theory of large deviations (and easy to deduce from the LDP) that for many interesting sets $A\subset \Ncal\times \Mcal $ one also has that $\P_N(({\rm Mi}^{\ssup N}, {\rm Ma}^{\ssup N})\in A)=\e^{-N \inf_A I(\cdot;t)(1+o(1))}$, for example for sets $A$ that are equal to the closure of their open kernel. There are choices of such sets that give the precise exponential rates of interesting  events, for instance the event that there are a given number of components larger than $Na$, for some $a>0$, or that a given component size appears with a certain least density, or that a given positive percentage of vertices are contained in components of a given range of sizes (e.g., in $\{1,\dots,R\}$  or in $\{R,\dots,\eps N\}$ or in $\{\eps N,\dots,N\}$), and certainly all kinds of combinations of such events.

From our main result, the LDP in Theorem \ref{thm-LDP}, a number of other LDPs follow via the contraction principle, according which if a random variable satisfies an LDP, so does its image under a continuous transformation; see \cite{DZ10}. Let us begin with the component size distribution of the microscopic part.

\begin{cor}[LDP for microscopic component size statistics]\label{cor-LDP_mi}
As $N\to\infty$,  ${\rm Mi}^{\ssup N}$ satisfies an LDP with speed $N$ and rate function $\Ical_{\rm Mi}(\cdot;t)\colon \Ncal \to[0,\infty]$, given by
\begin{equation}\label{Micro_rate_f}
\Ical_{\rm Mi}(\Lambda;t)= \inf_{\alpha\in\Mcal}I(\Lambda,\alpha;t)= I_{\rm Mi}(\Lambda;t)
 - (1-c_{\Lambda})\left(\log\frac{1-\e^{(c_{\Lambda}-1)t}} {1-c_{\Lambda}} - \frac{c_{\Lambda}t}{2}\right).
\end{equation}
\end{cor}

The first equality comes from the application of the contraction principle; while the second equality is purely analytical and it is checked in Lemma \ref{lem_min_alpha}. There it is seen that, given any $\Lambda\in \Ncal$, it is always optimal to have all the remaining mass $1-c_\Lambda$ in one single macroscopic component.

In the same way one can investigate the macroscopic part of the system.

\begin{cor}[LDP for macroscopic particles]\label{cor-LDP_ma}
As $N\to\infty$,  ${\rm Ma}^{\ssup N}$ satisfies an LDP with speed $N$ and rate function $\Ical_{\rm Ma}(\cdot;t)\colon \Mcal \to[0,\infty]$, given by
\begin{equation}\label{macro_rate_f}
 \begin{aligned}
\Ical_{\rm Ma}(\alpha;t)&= 
\inf_{\Lambda\in\Ncal}I(\Lambda,\alpha;t)\\
&= I_{\rm Ma}(\alpha;t)+(1-c_{\alpha})\Big(\frac t2-\log t\Big)+C_{\alpha,t}\Big (\log (tC_{\alpha,t})-\frac t2 C_{\alpha,t}\Big),
\end{aligned}
\end{equation}
where $C_{\alpha,t}= (1-c_{\alpha})\wedge \frac 1t$ (recall $c_\alpha= \int_{0}^1x\,\alpha(\d x)$).
\end{cor}

Again, only the second equality has to be checked; this is done in Lemma \ref{lem_min_lambda}. In contrast with the result above, here the optimal configuration $\Lambda^*$ depends on $\alpha\in\Mcal$, most heavily it depends on whether $1-c_\alpha\leq \frac 1t$ or not. Indeed, if $1-c_\alpha\leq \frac 1t$, then $c_{\Lambda^*}=1-c_{\alpha}$ (and no mesoscopic part arises). However, if $1-c_{\alpha}>\frac 1t$, then $c_{\Lambda^*}=\frac 1t$, and a non-trivial mesoscopic mass arises in the minimization; see Theorem~\ref{thm-phasetrans}. This peculiarity shows already a key difference between the cases $t\leq 1$ and $t>1$. Indeed, if $t\leq 1$ one cannot have any macroscopic mass distribution $\alpha$ such that $1-c_\alpha>\frac 1t$ and no difference in the minimizing strategy of the system can be seen. This is a first way to see the phase transition at $t=1$ from analytic properties of the rate function.

Now we come to the mesoscopic part of the particle configuration. This part comprises particle sizes on all the scales between finite and $O(N)$ and it seems unreasonable to consider an empirical measure for it. Instead, we consider only the total mass of this mesoscopic part. Let $\eps>0$ and $R\in\mathbb{N}$ be two auxiliary parameters, then we define the $(R,\eps)$-mesoscopic total mass as
\begin{equation}
\overline{\rm Me}^{\ssup N}_{R,\eps}=\frac 1N\sum_{i\colon R< S^{\ssup N}_i< \eps N}S^{\ssup N}_i.
\end{equation}
This is the number of vertices that are contained in components with a size between $R$ and $\eps N$.
The mesoscopic total mass in a strict sense arises after taking the limits $N\to\infty$, followed by $\eps\downarrow 0$ and $R\to\infty$, but this does not define a proper random variable.
However, it is possible to formulate an LDP in the $N\rightarrow \infty$ limit and then to study the rate function, $\Jcal_{\rm Me}^{\ssup{R,\eps}}$, as $\eps\downarrow 0$ and $R\to\infty$.
Additionally, the proof of Theorem~\ref{thm-LDP} shows that it is possible to define a coupled mesoscopic total mass $\overline{\rm Me}^{\ssup N}_{R_N,\eps_N}$, for any diverging sequence $R_N$ and vanishing sequence $\eps_N$. This is a well-defined random variable, it satisfies an LDP with speed $N$ and the rate function is the limit of $\Jcal_{\rm Me}^{\ssup{R,\eps}}$ when $\epsilon\searrow 0$ and $R\nearrow\infty$.

\begin{cor}[LDP for mesoscopic mass]\label{cor-LDP_me}
\begin{enumerate}
 \item For any $R\in\N$ and $\eps\in(0,1)$, as $N\to\infty$, $\overline{\rm Me}^{\ssup N}_{R,\eps}$ satisfies an LDP with speed $N$ and rate function $c\mapsto \Jcal_{\rm Me}^{\ssup{R,\eps}}(c;t)$, where
$$
\Jcal_{\rm Me}^{\ssup{R,\eps}}(c;t)=\inf\Big\{I(\Lambda,\alpha;t)\colon \sum_{k=1}^R k\lambda_k + \int_{\eps}^1x\,\alpha(\d x)=1-c\Big\}.
$$

\item For any $R_N\in\N$ and $\eps_N\in(0,1)$ such that $1\ll R_N< \eps_N N\ll N$, and $|\frac 1{\eps_N}\log\eps_N|\leq o(N)$, the coupled mesoscopic total mass $\overline{\rm Me}^{\ssup N}_{R_N,\eps_N}$ satisfies an LDP with speed $N$ and rate function
\begin{equation}\label{Jmesodef}
\begin{aligned}
\Jcal_{\rm Me}(c;t)&=\lim_{R\to\infty,\eps\downarrow 0}\Jcal_{\rm Me}^{\ssup{R,\eps}}(c;t)\\
&=(1-c)\Big ( \log(1-c)t-\frac{(1-c)t}{2}\Big)+\frac t2-\log t.
\end{aligned}
\end{equation}
The function $\Jcal_{\rm Me}(c;t)$ is strictly increasing in $c$, its minimum over $[0,1]$ is $\Jcal_{\rm Me}(0;t)=0$.
\end{enumerate}
\end{cor}

Corollary \ref{cor-LDP_me} part (1) is a simple consequence of the contraction principle, as the maps $\Lambda\mapsto \sum_{k=1}^R k\lambda_k$ and $\alpha\mapsto  \int_{\eps}^1x\,\alpha(\d x)$ are continuous. Assertion (2) follows as a byproduct of our proof of Theorem \ref{thm-LDP} in Section~\ref{sec-ProofLDP}. 

Hence, $\Jcal_{\rm Me}(\cdot;t)$ can rightfully be called the rate function for the mesoscopic total mass. Since it is positive everywhere outside $0$, we have the immediate consequence that the probability that any positive percentage of the vertices lies in mesoscopic components decays exponentially towards zero. This implies the convergence in probability of $\overline \Me_{R_N,\eps_N}^{\ssup N}$ towards zero with exponential decay of the probability of a decay by any positive amount. Interestingly, taking $R_N+1=\eps_N N\in\N$, we see that already just one mesoscopic size alone satisfies the same LDP as the entire $(R,\eps)$-mesoscopic total mass in the limit $R\to\infty$, $\eps\downarrow 0$. 
The condition $|\frac 1{\eps_N}\log\eps_N|\leq o(N)$ is not only a technical one, but implies that $\log N\ll \eps_N N$, taking care of the well-known fact that there are many clusters of size $O(\log N)$ in  the sparse Erd\H{o}s-R\'enyi random graph  that stem from an extreme-value statistics effect of the microscopic clusters. 

\subsection{Our results: the phase transition in the light of the LDP}\label{sec-phasetrans}

We now proceed with the study of the main phenomenon in the sparse Erd\H{o}s-R\'enyi random graph: the phase transition  of the emergence of a giant component. We will deduce it from our large-deviations rate functions from Section \ref{sec-Results}. The LDPs and the identification of their strict minimiser(s) lead to laws of large numbers for a number of random quantities. Indeed, it is a standard and simple fact from large-deviations theory that a random variable that satisfies an LDP with a rate function that contains precisely one minimizer converges in probability to that minimizer. We will exploit this fact to deduce laws of large numbers. As before, the parameter $t\in(0,\infty)$ will play the decisive role; recall that the connection probability $\frac{1}Nt_N$ of the graph $\Gcal(N,\frac{1}Nt_N)$ was picked as $t_N=t+o(1)$ as $N\to\infty$.

Consider the following functions of the total masses of the microscopic and macroscopic particles respectively:
\begin{equation}\label{JMIandJMa}
\Jcal_{\rm Mi}(c;t)=\inf_{\Lambda\in\Ncal(c)}\Ical_{\rm Mi}(\Lambda;t)\qquad\mbox{and}\qquad
\Jcal_{\rm Ma}(c;t)
=\inf_{\alpha\in\Mcal_{\mathbb{N}_0}((0,1];c)}\Ical_{\rm Ma}(\alpha;t),
\end{equation}
where $c\in[0,1]$. 
These two functions are not entirely analogous to $\Jcal_{\rm Me}(c;t)$ as rate functions for the total masses of the micro and the macro part, because the total masses both of ${\rm Mi}^{\ssup N}$ and ${\rm Ma}^{\ssup N}$ are equal to one. This is consistent with the fact that the contraction principle cannot be applied to total masses, as they are not continuous functions of the measures. However, they contain rather interesting information about the phase transition. 

\begin{theorem}[Microscopic total mass phase transition]\label{thm-phasetrans}
\begin{enumerate}
\item For any $c\in[0,1]$,
\begin{equation}\label{min_lambda}
\begin{aligned}
 \Jcal_{\rm Mi}(c;t)&=t c +(1-c)\log \frac{1-c}{1-\e^{t(c-1)}}+
\begin{cases}
c\log c- t c^2&\text{for }c<\frac 1t,\\
-\frac 1{2t}-\frac {t}{2}c^2-c\log t&\text{for }c\geq\frac 1t.
\end{cases}
\end{aligned}
\end{equation}
Moreover, $\Jcal_{\rm Mi}(c;t)=\Jcal_{\rm Ma}(1-c;t)$.
\item For $c\in(0,1]$, the minimum of $\Ncal(c)\ni\Lambda\mapsto \Ical_{\rm Mi}(\Lambda;t)$ is attained precisely at $\Lambda^*(c;t)\in\Ncal(c)$ given by
\begin{equation}\label{lambdamin}
\lambda^*_k(c;t)=\frac{k^{k-2}c^k t^{k-1}\e^{-ctk}}{k!},\qquad k\in\N,
\end{equation}
and the minimum of the function $c\mapsto  \Jcal_{\rm Mi}(c;t)$ is attained precisely at $c=1$ with value $\Jcal_{\rm Mi}(1;t)=0$. Therefore the infimum 
\begin{equation}
\inf_{(\Lambda,\alpha)\in\Ncal\times\Mcal}I(\Lambda,\alpha;t)\label{total_inf}
\end{equation}
is attained at $(\Lambda,\alpha)=(\Lambda^*(1;t),\mathbf{0})$, where $\mathbf{0}=(0,0,\dots)$.

\item For $t\in(1,\infty)$, the minimum of the function $c\mapsto \Jcal_{\rm Mi}(c;t)$ is attained at $c=\beta_t$ where  $\beta_t\in(0,t)$ is the smallest positive solution to 
\begin{equation}\label{equilibrium}
\log \beta_t=t\beta_t-t.
\end{equation}
The infimum in \eqref{total_inf} is attained precisely at $(\Lambda,\alpha)=(\Lambda^*(\beta_t;t),(1-\beta_t,0,0,\dots))$.
\end{enumerate}
\end{theorem}

The proof is found in Section \ref{sec-Proofphasetrans}.
 
The two different cases in \eqref{min_lambda} refer to the cases that the first minimum in \eqref{JMIandJMa} is attained or not. Indeed, for 
$c\leq \frac 1t$, the function $\Ical_{\rm Mi}(\cdot; t)$ is minimized in an optimal $\Lambda^*$ with $c_{\Lambda^*}=c$. However, for $c>\frac 1t$ this is not possible, but only minimizing sequences can be found that achieve a total mass of $\frac 1t$ in the microscopic measure and displace the remaining mass $c-\frac 1t$ to the mesoscopic part. This shows that the phase transition originates from the impossibility of picking an optimal microscopic configuration $\Lambda^*$ if its total mass $c_{\Lambda^*}$ is required to be too large; the threshold being $1/t$. If this is exceeded, then a minimization can be done only with the help of some non-trivial mesoscopic part. As we mentioned above, the macroscopic configuration is always minimized in one single giant component.

The same effect is seen in $\Ical_{\rm Ma}(c;t)$, where first an optimization over $\Lambda$ with $c_\Lambda\leq 1-c$ is performed, and such a balance between microscopic and mesoscopic mass can pop out if $1-c$ is large enough. Subsequently, optimizing over $\Mcal_{\N_0}((0,1];c)$ is straightforward.  The equality $\Jcal_{\rm Mi}(c;t)=\Jcal_{\rm Ma}(1-c;t)$ follows from this.

In Theorem \ref{thm-phasetrans}(2) and (3) we see the different behaviour for subcritical, respectively supercritical $t$ in terms of the microscopic configuration. Note that this configuration is actually given by 
\[ 
k\lambda^*_k(c;t)=c{\rm Bo}_{ct}(k),
\] 
where ${\rm Bo}_\mu$ is the Borel distribution with parameter $\mu\in[0,1]$. We see that such an optimal $\Lambda^*(c;t)$ cannot be found if $c>\frac 1t$, and this is an admissible total mass only when $t>1$, marking the threshold between subcritical and supercritical regime (otherwise, there is no relevant case distinction as to the value of $c$). In this way, the Borel distribution appears as the natural minimizer of the microscopic part of the rate function. 

In earlier work (see \cite{Pit90}), the appearance of the Borel distribution in this context came from the observation that ${\rm Bo}_\mu$ is the distribution of the total progeny of a Galton--Watson tree with offspring that is Poisson-distributed with parameter $\mu$. The characterisation of the emerging cluster-size distribution $\Lambda^*$ was based on an approximation of the connected subgraphs by such trees and counting the total number of trees of a given size in the graph. This approximation argument was extended to a large-deviation setting in \cite{BorCap15}, see Section \ref{subsect:ERRG}.

Theorem \ref{thm-phasetrans} characterises the well-known phase transition of the ermergence of a giant component at $t=1$ in terms of a natural notion that is familiar to statistical physics: as a non-analyticity of the limiting free energy for the total mass of the microscopic configuration, which is equal to the infimum of $\Jcal_\Mi(\cdot;t)$. Indeed, this function is zero in $[0,1]$, but positive in $(1,\infty)$.

Another characterisation of this phase transition is in terms of a law of large numbers. Indeed, combining Theorem~\ref{thm-phasetrans} with the LDP in Theorem~\ref{thm-LDP} one has
$$
\big({\rm Mi}^{\ssup N},{\rm Ma}^{\ssup N}\big)\overset{N\to\infty}{\Longrightarrow}
\begin{cases}
        (\Lambda^*(1;t),\mathbf{0})&\mbox{if }t\leq 1,\\
        (\Lambda^*(\beta_t;t),(1-\beta_t,0,\dots))&\mbox{if }t\geq 1.
\end{cases}
$$
In words, this means that, for any $k\in\N$,  $\frac 1N$ times the number of components of size $k$ converges to $\lambda_k^*(1;t)$ in the sub-critical regime  and to $\lambda_k^*(\beta_t,t)$ in the supercritical regime, while there is no macroscopic component in the first regime and there is precisely one macroscopic cluster of cardinality $\sim N(1-\beta_t)$ in the second. All these statements are in the sense of convergence in probability, and the probability of a deviation by any positive amount decays even exponentially in $N$.

One also sees that the cut-off versions of the total masses, $\sum_{k=1}^R k {\rm Mi}_k^{\ssup N}$ and $\int_{[\eps,1]} x\, {\rm Ma}^{\ssup N}(\d x)$, converge towards the respective cut-off versions of the limits, and their limits as $R\to\infty$ and $\eps\downarrow 0$ are $(1,0)$ for $t\leq 1$ and $(\beta_t,1-\beta_t)$ for $t\geq1$.

\subsection{Related works on LDPs for Erd\H{o}s-R\'enyi graphs}\label{subsect:ERRG}
Despite the extensive literature on the Erd\H{o}s-R\'enyi graph, there are not many results about large deviations in the sparse regime. Here we summarize, to the best of our knowledge, the existing results and how they relate to our work.
 
Two LDPs for the size of the largest component and for the number of isolated vertices  have been derived in \cite{OCon98}. These are two quantities that are obviously functionals of our measures $\Ma^{\ssup N}$, respectively of $\Mi^{\ssup N}$. Indeed, the largest component is the mass of the largest atom of $\Ma^{\ssup N}$, and the number of isolated vertices  is equal to $N$ times $\Mi_1^{\ssup N}$. Both these two functionals are continuous, such that the contraction principle applies. The approach of \cite{OCon98} is a simplified version of our comprehensive approach for the joint distribution of all the component sizes, and consequently it leads to formulas for the rate functions that are contractions of our rate function, which is straightforward to see. This explains also the remark made about the lack of convexity of the rate function in~\cite{OCon98}. Indeed, contraction often ruins convexity and contracted rate functions are rarely convex. Hence, our work includes the results of \cite{OCon98}.

A route that is inspired by statistical physics is taken in \cite{EnMoRe04}, where the distribution of the random graph is tilted with a parameter $q>1$ raised to the power of the number of components, properly normalized. The analysis of the free energy of the corresponding partition sum is carried out there. Via the well-known Laplace dualism, the results is essentially equivalent to an LDP for the number of components. This functional is equal to the continuous functional $\Mi^{\ssup N}(\N)$ in our setting. The pecularity of  \cite{EnMoRe04} is that this model is put into relation with the $q$-state Potts model in the limit $q\downarrow 1$ via diagrammatic expansion techniques. In particular, they derive limiting formulas for the size of the giant component, the degree distributions inside and outside the giant component, and the distribution of small component sizes.

We already mentioned that registering each component as a {\em subgraph} (rather than only as its size) would give \emph{a priori} a much more detailed description, at least for any fixed $N$.  However, in the limit as $N\to\infty$, in the LDP regime that we consider, only few subgraph configurations survive: the microscopic components survive only as spanning trees, and only those macroscopic components survive that have an excess of edges of order $\Uptheta(N)$. The first has been carried out in \cite{BorCap15}, the second in \cite{Puh05}.

Indeed, the macroscopic part of our LDP is covered in \cite{Puh05}. The author gives an LDP for the joint distribution of the total number of components, the sequence of the sizes of macroscopic ones, and the sequence of  corresponding numbers of the excess edges appended with zeros. Therefore the contraction of this LDP to the total number of components and the macroscopic sizes (see \cite[Corollary 2.1]{Puh05})  is equal to the contraction of our LDP from Theorem \ref{thm-LDP} to the LDP for $\left(\sum_{k}\Mi_k^{\ssup N},\Ma^{\ssup N}\right)$. The same is of course true, when considering exclusively macroscopic sizes (compare \cite[Corollary 2.2]{Puh05} with Corollary \ref{cor-LDP_ma}). The approach in \cite{Puh05} goes along a very different route, involving recursive formulas for the graphs $\Gcal(N,\frac 1N t_N)$ if $N$ increases, and consequently the form of the rate function derived there is pretty different from ours; it involves an additional minimization procedure. It would require some work to analytically check that it is identical to ours.

In \cite{BorCap15}, an LDP for the empirical measure of all the components rooted at the vertices, is derived with an explicit rate function. The topology used there comes from a distance that looks only at intersections of graphs with bounded sets, so it can detect only microscopic components. In this way it is contained in our results. However, \cite{BorCap15} considers the components as graphs, not only as sizes, and gets therefore a much more detailed picture. Moreover, the object described in \cite{BorCap15} is a size-biased version of our microscopic measure, since we are counting components of a certain size, while they consider the component containing each vertex and therefore counting a certain component proportionally to the number of vertices it contains. 
Hence, the LDP of \cite{BorCap15} contains the microscopic part of our LDP (Corollary \ref{cor-LDP_mi}) via the contraction principle, but a certain normalization has to be performed to actually compare the two objects.  However, let us notice that \cite[Theorem 1.8]{BorCap15} shows that the rate function takes the form of a sort of relative entropy with respect to a Galton Watson tree with Poisson offspring distribution (plus additional constants). This form is shown also in our contracted rate function from Corollary \ref{cor-LDP_mi}, which in \eqref{eq_micro_entropy} we rewrite in terms of a relative entropy with respect to a distribution related to the Borel distribution (which is the distribution of the total progeny of precisely a Galton Watson tree with Poisson offspring). Also in this case, one sees that the Borel distribution appears as a size-biased version of our reference distribution. 


Under assumptions that imply that the connection probability of the Erd\H{o}s-R\'enyi graph $\Gcal(N,p)$ satisfies $p \gg N^{-\frac12}$, recent progress has been made on the upper tails of sub-graph counts \cite{Chat16,Aug18, CooDem18}.

In the case of dense graphs, that is, for $\Gcal(N,p)$ with fixed $p\in(0,1)$, there is a complete treatment thanks to Chatterjee and Varadhan~\cite{ChaVar11}, see~\cite{Chat16b} for an overview. This regime is rather different from the sparse regime, since a proper formulation of the relevant limiting objects requires a abstract setting evolving around the notion of a graphon.

\subsection{Application to coagulation models}\label{sec-Coagulation}

Our interest in this research came from the desire to understand dynamical particle systems with coagulation in the large-system limit. It turned out that one of the most prominent (and most simple) models, the {\it Marcus--Lushnikov model of coagulation}, see \cite{Mar68,Gil72,Lus78b}, admits a one-to-one correspondence to the component sizes of the Erd\H{o}s-R\'enyi random graph that we study in this paper.  This coagulation process is a continuous-time Markov process of vectors of particle masses $(S_i^{\ssup N}(t))_{i\in\{1,\dots,n_t\}}\in(\N_0)^{n_t}$ at time $t\in[0,\infty)$, arranged in descending order, precisely as in \eqref{componentsizes}, where $n_t$ is the number of particles at time $t$. This process is specified by the initial configuration, which we take in the monodisperse case, i.e., $S_i^{\ssup N}(0)=1$ for all $i=1,\dots,N=n_0$, and by the transition mechanism, which is given in terms of a symmetric, non-negative {\it coagulation kernel} $K_N\colon \N\times\N\to[0,\infty)$. That is, we start with $N$ particles of unit mass at time 0, and in the course of the process, each (unordered) pair of particles with respective masses $m,\widetilde m\in\N$ coagulate to a particle of mass $m+\widetilde m$ with rate $K_N(m,\widetilde m)$, independently of all the other pairs of particles. 

The important special case of the {\em multiplicative kernel}, $K_N(m,\widetilde m)=m\widetilde m /N$ has the two interesting features: (1) it can be mapped onto the Erd\H{o}s-R\'enyi random graph that we study in this paper, and (2) it exhibits an interesting {\it gelation phase transition} in the limit $N\to\infty$ at time $t=1$, because a {\em gel}, i.e., a particle of macroscopic size, appears. Indeed, for this process, 
it turned out in the review \cite{Ald99} that, for fixed $t\in[0,\infty)$, the distribution of the family $(S_i^{\ssup N})_{i\in\{1,\dots,n\}}$ of the component sizes of $\Gcal(N,\frac 1N t_N)$ defined in \eqref{componentsizes} with $\frac 1Nt_N=1-\e^{-t/N}$ is identical to the family $(S_i^{\ssup N}(t))_{i\in\{1,\dots,n_t\}}$ of particle masses in the multiplicative coalescent at time $t$.
This correspondence was not mentioned in \cite{Buf90}, but was discussed one year later in \cite{Buf91}, which highlights the connection between gelation in the coagulation process and the phase transition given by the formation of a giant connected component in the Erd\H{o}s-R\'enyi random graph \cite{Erd60}.

Hence, the results of this paper also recover the gelation phenomenon in rather explicit terms through a LDP in terms of the microscopic, mesoscopic and macroscopic parts, in the same way as we explained in the above sections for the Erd\H{o}s-R\'enyi random graph.

Smoluchowski introduced a (deterministic) ODE model for the concentrations of coagulating particles in the course of his work on Brownian motion \cite{Smol16b}. Indeed, it is reasonable to assume that $l_k(t) = \lim_{N\rightarrow \infty}\frac 1N\# \{ \text{particles of size } k \text{ at time } t\}$ exists under suitable conditions, see \cite{Nor99, LaMi04, MeNo14}. These limits satisfy
\begin{equation}\label{e:smol}
\frac{\d }{\d t}l_k(t) = \frac12 \sum_{\substack{m, \widetilde m \colon \\ m + \widetilde m=k}}l_m(t) l_{\widetilde{m}}(t) K(m,\widetilde{m})
-l_k(t)\sum_{m}l_m(t) K(k,m),
\qquad k\in\N,
\end{equation}
where $K(m,\widetilde m) = \lim_{N\rightarrow\infty} N K_N(m,\widetilde m)$ is the limiting coagulation kernel (in our case, $K(m,\widetilde m)=m\widetilde m$). This is the famous {\em Smoluchowski equation}. Intuitively, the positive terms on the right-hand side of \eqref{e:smol} take into account that the fraction of particles of mass $k$ increases if a particle of mass $m$ and one of mass $\widetilde m$ (with $m+\widetilde m=k$) merge and this happens with rate $K(m,\widetilde m)$. On the other hand, the negative term describes that a particle of mass $k$ can coagulate with particles of any size $m$ with rate $K(k, m)$ and this is why it involves an infinite sum (over all $m\in \N$).
One can check for $t\leq 1$ that $\Lambda^*(1;t)$  appearing in Theorem \ref{thm-phasetrans} is the exact solution of \eqref{e:smol}, the Smoluchowski equation, also given in \cite[Table 2]{Ald99}. As a consequence, the above mentioned gelation phase transition as well as the solution of the Smoluchowski ODE are also clear from our results in Sections \ref{sec-Results} and \ref{sec-phasetrans} and receive therefore a new interpretation in terms of combinatorial structures.

In the light of the process character of the Marcus--Lushnikov coagulation model, it will be desirable to derive a pathwise version of the LDP of Theorem \ref{thm-LDP}. This will require a version of that theorem which starts from an arbitrary configuration rather than from $S_i^{\ssup N}(0)=1$. This may also be interesting for the time-dependent version of the Erd\H{o}s-R\'enyi graph $(\Gcal(N,1-\e^{t/N}))_{t\in[0,\infty)}$, but not as natural as for the Marcus--Lushnikov model. Another aspect that makes it particularly interesting for coagulation models is the availability of alternative methods in the spirit of Wentzell--Freidlin theory to derive pathwise LDPs for coagulation models, see \cite{Mie17}. Let us also mention that, in the renowned paper \cite{Ald97}, time is expanded around the critical value $t=1$, and the mesoscopic components of the graph are compared to a stochastic process known as the \emph{multiplicative coalescent}. Although we allow for fluctuations around $t$ (our LDP holds for any sequence $t_N\sim t$), we cannot capture this regime around $t=1$, since we expect a LDP for mesoscopic particles to hold on a different scale. 
Another natural direction of our future research is an extension to the case of an {\it inhomogeneous} Erd\H{o}s-R\'enyi graph as introduced in \cite{BoJaRi07}. We will defer future work to these questions.

\subsection{Comparison to Bose-Einstein condensation without interaction}\label{sec-BEC}

Our large-deviations approach to the Marcus--Lushnikov models shows remarkable similarities to another well-known phase transition in a non-spatial model, the non-interacting Bose gas. Here the situation is similar in that the gas can be conceived as a joint distribution of $N$ particles that are randomly grouped into smaller units, called {\em cycles}, which can become arbitrarily large. The natural question is then, under what circumstances do macroscopic cycles arise. An explicit answer in terms of a large-deviations analysis has been given in \cite{A08}, where the transition, the famous {\em Bose--Einstein Condensation (BEC)} in dimensions $d\geq 3$, is derived from the minimization of the rate function, in a way analogous to that in our Theorem~\ref{thm-phasetrans}. 
The two phase transitions differ in that the BEC transition is of \emph{saturation type}, while the gelation transition is not.

For the non-interacting  Bose gas in the thermodynamic limit at temperature $1/\beta\in(0,\infty)$ with particle density $\rho\in(0,\infty)$ the partition function is given by
$$
Z_{\Lambda_N}^{\ssup{\beta}}
=\sum_{(\ell_k)_{k\in\N}\in\N_0^{\N}\colon \sum_k k \ell_k=N}\prod_k\frac{N^{\ell_k}}{\ell_k!\,k^{\ell_k}}[\rho (4\pi\beta k)^{\frac d2}]^{-\ell_k},
$$
where $\Lambda_N$ is the centred box in $\R^d$ with volume $N/\rho$.  The free energy per particle is then 
$$
f(\beta,\rho)=\lim_{N\to\infty}\frac 1N\log Z_{\Lambda_N}^{\ssup{\beta}}=-\inf_{\Lambda\in\Ncal(\rho)}I(\Lambda),\qquad\mbox{where}\qquad
I(\Lambda)=\sum_k \lambda_k\log\frac{\lambda_k k}{(4\pi\beta k)^{\frac d2}\e}.
$$
For the Erd\H{o}s-R\'eny graph  $\Gcal(N,\frac 1N t_N)$, the  equivalent quantity is the rate function $\Ical_{\rm Mi}$ from \eqref{Micro_rate_f}.
The key difference between the rate functions is that  only $\Ical_{\rm Mi}$ contains terms in the total mass of microscopic components, $c_\Lambda$.
This reflects the fact that the giant component makes a significant contribution to the rate function in the graph  model, but the condensate in the non-interacting Bose gas does not.

The respective minimisers of $I_{\rm Mi}$ and $I$ are
$$
 k \lambda^{\ssup{\rm ML}}_k(c;t)=\frac 1t\,\frac{(c t\e^{-ct})^k}{k^{1-k}\,k!} \sim
   \frac{1}{ \sqrt{2\pi}t} \frac{\left(ct\e^{-ct+1}\right)^k}{ k^{3/2}}
\quad\text{and}\quad 
k \lambda^{\ssup{\rm BEC}}_k(\alpha;\beta)=\frac{1}{\rho(4\pi\beta)^{\frac d2}}\frac{\e^{-\alpha k}}{k^{\frac d2}},
$$
where $c$ and $\alpha$ control the values of $\sum_k k\lambda_k$.

The crucial parameters are $t$ for the graph model and the inverse temperature $\beta$ for the Bose gas.
Both models have a trivial upper bound for the total microscopic mass, $\sum_k k \lambda_k$, namely one.
One additional upper bound arises in each model from the optimisation of the rate function with respect to the $\lambda_k$, but these are not relevant, until $t$ respectively $\beta$ rises to its critical value.
For the graph model this bound is $1/t$, because $\sum_k \frac{(c t\e^{-ct})^k}{k^{1-k}\,k!} \leq 1$ for all $ct\in (0,\infty)$, and the summands take their maxima at $ct=1$, when they correspond to the Borel probability distribution with parameter 1.
For $ \Lambda^{\ssup{\rm BEC}}$ this bound is $\rho^{-1}(4\pi\beta)^{-d/2}\sum_k k^{-\frac d2}$.
At this point we see a difference between the two models, because the total microscopic mass in the Bose gas remains on this bound as $\beta$ rises further, while for the graph model it immediately drops strictly below the bound.
This explains why BEC is known as a saturation phase transition, but this description cannot be applied to gelation.

\section{Preparations for the proof of the LDP}\label{sec-distributions}

\noindent We consider the Erd\H{o}s-R\'enyi graph $\Gcal=\Gcal(N,p)$ under the corresponding probability measure $\P_{N,p}$. In Section \ref{sec-steps}, we derive  an explicit formula for the distribution of the empirical measure of the component sizes $S_i^{\ssup N}$ in terms of connectivity probabilities for (smaller) Erd\H{o}s-R\'enyi random graphs. Furthermore, we prepare in Section \ref{sec:asympt_prob} for the asymptotic analysis f or $p=\frac 1N t_N$ with $t_N=t+o(1)$ by recalling from \cite{Ste70} some estimates and asymptotics for this connectivity probability.

\subsection{The joint distribution of the component sizes}\label{sec-steps}

An important quantity is 
\begin{equation}\label{mudef}
\mu_{k}(p)=\P_{k,p}\big(\Gcal\text{ is connected}\big), \qquad k\in\N,p\in[0,1].
\end{equation} 
We will be concerned with this quantity for fixed $k$, but with connection probability $p=\frac{1}Nt_N$, in the limit $N\to\infty$.

We define the state space of the collection of component sizes as
\begin{equation}\label{state_space_N}
E_N=\Big\{(s_i)_{i\in\{1,\dots,n\}}\in\N_0^n\colon n\in\N, s_1\geq s_2\geq \dots\geq 0,\,\sum_{i=1}^n s_i=N\Big\}. 
\end{equation}
To each element $(s_i)_i$ of the state space $E_N$, 
we associate a unique element of the space 
\begin{equation}\label{state_space_2}
\Ncal_N=\Big\{\ell=(\ell_k)_k\in\N_0^\N\colon \sum_k k\ell_k=N\Big\},
\end{equation}
where for each $k$, $\ell_k$ is the number of indices $i$ such that $s_i=k$.  The map $(s_i)_i\mapsto \ell $ is a bijection and in the following we refer to configurations equally in terms of $(s_i)_i$ or $\ell$. 

We denote by $\Pcal_N$ the set of all partitions of $[N]=\{1,\dots,N\}$. We write $B_k(\pi)$ for the number of sets in $\pi\in \Pcal_N$ with cardinality $k$. 
Then we can describe the joint distribution of all the component sizes of $\Gcal(N,p)$ as follows.

\begin{lemma}For any $p\in[0,1]$, $N\in\N$ and every $(s_i)_i\in E_N$,
\begin{equation}\label{coagprob}
\P_{N,p}\big((S_i^{\ssup N})_i =(s_i)_i\big)=\#\{\pi\in\Pcal_N\colon B_k(\pi)=\ell_k\,\forall k\}\times\Big(\prod_{i}\mu_{s_i}(p)\Big)\times \Big(\prod_{i\not = j}(1-p)^{\frac 12 m_i\,m_j}\Big).
\end{equation}
\end{lemma}

\begin{proof} A set $A\subset\{1,\dots,N\}$ of vertices is a connected component in the graph $\Gcal(N,p)$ if and only if (1) no bond between any vertex in $A$ and any vertex outside has been put, and (2) the subgraph formed out of the vertices in $A$ and all the bonds between any two vertices in $A$ is connected. This has probability $(1-p)^{|A|\,|A^{\rm c}|} \times \mu_{|A|}(p)$. Applying this reasoning to $A^{\rm c}$ and describing the next component, and iterating this argument, shows that the product of the two products on the right-hand side of \eqref{coagprob} is equal to the probability, for a given partition $\pi$ with $\ell_k$ sets of size $k$ for any $k$, that the components of $\Gcal(N,p)$ are precisely the sets of $\pi$. Since this probability depends only on the cardinalities, the counting term completes the formula.
\end{proof}

Now we rewrite the right-hand side of \eqref{coagprob} in terms of the empirical measure of $(s_i)_i$, i.e., of the numbers $\ell_k$ of indices $i$ such that $s_i=k$. Introduce the event
\begin{equation}\label{statisticsevent}
A_{N}(\ell)=\bigcap_{k\in\N}\{\#\{i\colon S_i^{\ssup N}=k\}=\ell_k\},\qquad\ell=(\ell_k)_{k\in\N}\in \N_0^\N.
\end{equation}

\begin{cor}\label{cor-probAN} For any $p\in[0,1]$, $N$ and any $\ell=(\ell_k)_k\in\N_0^\N$ satisfying $\sum_k k\ell_k=N$,
\begin{equation}\label{rewriteprobAN}
\P_{N,p}(A_{N}(\ell))=N!\prod_k\frac{\mu_{k}(p)^{\ell_k}(1-p)^{\frac 12 k(N-k)\ell_k}}{k!^{\ell_k} \, \ell_k!}.
\end{equation}
\end{cor}

\begin{proof} Note that the last product on the right-hand side of \eqref{coagprob} can also be written as $\prod_i (1-p)^{\frac N2 m_i(N-m_i)}$. Hence, if $\ell_k$ is equal to the number of $i$ such that $s_i=k$ for any $k$, then the product of the last two products can be written as 
 $$
 \prod_k \Big(\mu_k(p))^{\ell_k}(1-p)^{\frac 1{2}k(N-k)\ell_k}\Big).
 $$
 
The counting term is easily identified as
$$
\#\{\pi\in\Pcal_N\colon \#\{A\in\pi\colon |A|=k\}=\ell_k\,\forall k\}=\frac{N!}{\prod_{k} (k!^{\ell_k}\, \ell_k!)}.
$$
Substituting ends the proof.
\end{proof}

(To avoid confusion, we note that there is a typographical error in Section~4.5 of \cite{Ald99}, where the factor of $\frac 12 $ is missing in the exponent of \eqref{rewriteprobAN}.)

\subsection{The probability of being connected} \label{sec:asympt_prob} 

Our analysis of \eqref{rewriteprobAN} will depend crucially on an analysis of $\mu_{k}(\frac 1N t_N)$.
The next two lemmas collect results from \cite[Lemma 1\&2, Theorem 1]{Ste70}.

\begin{lemma}[Bounds and asymptotics for $\mu_{k}(\frac 1N t_N)$, \cite{Ste70}] \label{prob_finite_size}
For any $p\in[0,1]$, $N\in\mathbb{N}$ and any $k\leq N$,
\begin{equation}\label{Upp_bound_k-1_edges}
(1-p)^{\smfrac {1}{2}(k-1)(k-2)}\leq \frac{\mu_{k}(p)}{k^{k-2}p^{k-1}}\leq 1.
\end{equation}
In particular, if $p=\frac 1N t_N$ with $t_N=t+o(1)$ and $k=o(\sqrt{N})$,
\begin{equation*}
\mu_{k}(\smfrac 1N t_N)=k^{k-2}\Big(\frac tN\Big)^{k-1}(1+o(1)),\qquad N\to\infty.
\end{equation*}
\end{lemma}
The expression for the upper bound in \eqref{Upp_bound_k-1_edges} appears to be present (using somewhat applied chemical language) in~\cite[equation (5)]{Flo41b}.

The following is an alternative upper bound for $\mu_{k}(\frac 1N t_N)$, which will be required for macroscopic components, together with an asymptotic result for the connection probability in the so-called sparse case, where the bond probability is proportional to the inverse of the size of the graph.

\begin{lemma}[\cite{Ste70}]\label{mu_asy_gen}
For all $p\in[0,1]$ and $k\in\N$
\begin{equation*}
\mu_k(p) \leq  \big(1-\e^{k q}\big)^{k-1},
\end{equation*}
with $q=\log (1-p)$.
Moreover, for $\alpha\in(0,1)$ and $t\in(0,\infty)$ and a sequence $t_N=t+o(1)$, as $N\rightarrow \infty$,
\begin{equation}\label{eq:macro_asymp}
\mu_{\lfloor \alpha N\rfloor}(\smfrac 1Nt_N)=\Big(1-\frac{\alpha t}{\e^{\alpha t}-1}\Big)\big(1-\e^{-t\alpha}\big)^{\alpha N}(1+o(1)).
\end{equation}
\end{lemma}

\section{Proof of the LDP}\label{sec-ProofLDP}

\noindent In this section we prove the main result of this paper, the large-deviations principle in Theorem \ref{thm-LDP}. Again, we fix the parameter $t\in(0,\infty)$ and a sequence $t_N=t+o(1)$ and consider the Erd\H{o}s-R\'enyi graph $\Gcal(N,\frac 1Nt_N)$ with probability measure $\P_{N,\frac 1Nt_N}$. 

Recall the topological remarks on the two state spaces $\Ncal$ and $\Mcal$ from Section \ref{sec-Micromacro}. The metrics $d$ on $\Ncal$ and  $D$ on $\Mcal$, defined by
\begin{equation}\label{metrics}
d(\Lambda,\widetilde \Lambda)=\sum_{k=1}^{\infty}2^{-k}\,|\lambda_k-\widetilde\lambda_k|\qquad\mbox{and}\qquad D(\alpha,\widetilde \alpha)=\sum_{i=1}^{\infty}2^{-i}|\alpha_i-\tilde\alpha_i|,
\end{equation}
induce the respective topologies of pointwise and vague convergence. We write $B_\delta(\Lambda)$, respectively $B_\rho(\alpha)$, for the closed $\delta$-ball around $\Lambda$, respectively for the closed $\rho$-ball around $\alpha$.
Since 
the rate function $I(\cdot;t)$ is lower semicontinuous in $\Ncal\times\Mcal$ and the space is compact, we know that it is a good rate function (i.e., its level sets are not only closed but also compact). Therefore, a weak LDP implies our main result, the LDP in Theorem~\ref{thm-LDP}, and  it will be sufficient to prove the following.

\begin{prop}\label{Prop-mainLDP}
For any $(\Lambda,\alpha)\in \Ncal\times \Mcal$,
\begin{equation}\label{keyprob}
\lim_{\delta,\rho\downarrow 0}\lim_{N\rightarrow \infty}\frac 1N \log
\P_{N,\frac 1Nt_N}\big({\rm Mi}^{\ssup N}\in B_\delta(\Lambda),\, {\rm Ma}^{\ssup N}\in B_\rho(\alpha)\big)= -I(\Lambda,\alpha;t).
\end{equation}
\end{prop}

We split the proof of Proposition \ref{Prop-mainLDP} in several lemmas and finish it at the end of Section \ref{sec-ProofLDP}. We start by bounding the cardinality of $\Ncal_N$.

\begin{lemma}\label{lem_card_N}
 Let $\Ncal_N$ be as defined in \eqref{state_space_2}, then 
\[|\Ncal_N|=\e^{o(N)},\qquad N \to\infty.
\]
\end{lemma}
\begin{proof}
The following is an argument  in \cite{A08}. For any $\ell\in\Ncal_N$, the set $H( \ell)=\{k\in\N\colon \ell_k>0\}$ has no more than $2\sqrt N$ elements, since
$$
N=\sum_{k\in H(\ell)}k\ell_k\geq \sum_{k\in H(\ell)}k\geq \sum_{k=1}^{|H(\ell)|} k=|H(\ell)|\frac 12(|H(\ell)|-1).
$$
Hence,
$$
\begin{aligned}
|\Ncal_N|&= \Big|\Big\{(\ell_k)_k\in\N_0^{\N}\colon \sum_k k\ell_k=N, |H(\ell)|\leq 2 \sqrt N\Big\}\Big|\\
&\leq \sum_{H\subset[N]\colon |H|\leq 2\sqrt N}\Big|\Big\{(\ell_k)_{k\in H}\in\N^{|H|}\colon \sum_{k\in H}k \ell_k=N\Big\}\Big|\\
&\leq \sum_{H\subset[N]\colon |H|\leq 2\sqrt N}\Big|\{(L_k)_{k\in H}\in\N^{|H|}\colon \sum_{k\in H}L_k=N\Big\}\Big|
\leq \sum_{h=1}^{\lfloor 2 \sqrt N\rfloor} \binom N{h}\binom {N+h} {h}\\
&=\e^{o(N)}.
\end{aligned}
$$
\end{proof}

Thanks to Lemma \ref{lem_card_N}, it will be sufficient to  get estimates on  $\P_{N,\frac 1N t_N}(A_{N}(\ell))$, for any $\ell\in \Ncal_N$ close enough to a fixed $(\Lambda,\alpha)\in \Ncal\times\Mcal$.
The strategy is to divide the terms in the product representation from Corollary \ref{cor-probAN} into three groups, which we refer to as micro-, meso- and macroscopic, because they take into account the contribution of, respectively, micro-, meso- and macro- components. 
 We fix two increasing sequences $R_N$ and $\eps_N N$ in $\N$ such that $R_N\nearrow \infty$, $\eps_N \downarrow 0$ and $R_N<\eps_N N$. We write
\begin{equation}\label{ProbANsplit}
\P_{N,\frac 1N t_N}(A_{N}(\ell))=N! \times F_{\rm Mi}(\ell)\times F_{\rm Me}(\ell)\times F_{\rm Ma}(\ell),
\end{equation}
where 
$$
F_{\rm Mi}(\ell)=\prod_{k=1}^{R_N} z_k(\ell),\qquad F_{\rm Me}(\ell)=\prod_{R_N<k\leq \eps_N N}z_k(\ell),\qquad F_{\rm Ma}(\ell)=\prod_{\eps_N N <k\leq N} z_k(\ell),
$$
and 
$$
z_k(\ell)=\frac{\mu_k(\frac 1N t_N)^{\ell_k}(1-\frac{t_N}N)^{\frac 1{2} k(N-k)\ell_k}}{k!^{\ell_k} \, \ell_k!}.
$$
Let us set 
\begin{equation}\label{cDef}
c_\mathrm{Mi}(\ell/N)=\frac 1N\sum_{k=1}^{R_N}k\ell_k,\qquad
c_\mathrm{Me}(\ell/N)=\frac 1N\sum_{R_N<k\leq\eps_N N} k\ell_k,\qquad 
c_\mathrm{Ma}(\ell/N)=\frac 1N\sum_{\eps_N N < k\leq N}k\ell_k.
\end{equation}
Note that the sum of these three terms is equal to one. For the factor $N!$, we use Stirling's formula $N!=(\frac N\e)^N\e^{o(N)}$ so that uniformly in $\ell \in \Ncal_N$
\begin{equation}\label{StirlingAppl}
N!= \left(\frac{N}{\e}\right)^{N c_{\rm Mi}(\ell/N)}\left(\frac{N}{\e}\right)^{N c_{\rm Me}(\ell/N)}\left(\frac{N}{\e}\right)^{N c_{\rm Ma}(\ell/N)}\,\e^{o(N)},\qquad N\to\infty.
\end{equation}
We will consider a cut-off version of the distances $d$ and $D$ introduced in \eqref{metrics}, as follows:
\begin{equation*}
d_R(\Lambda,\widetilde \Lambda)=\sum_{k=1}^{R}2^{-k}\,|\lambda_k-\widetilde\lambda_k|\qquad\mbox{and}\qquad D_\epsilon(\alpha,\widetilde \alpha)=\sum_{i=1}^{\infty}2^{-i}|\alpha_i-\tilde\alpha_i|\mathds{1}(\alpha_i\vee \tilde\alpha_i\geq \epsilon N ),
\end{equation*}
such that for a given $\ell\in \Ncal_N$ we can measure simultaneously the distance of its microscopic part from $\Lambda$ and its macroscopic part from $\alpha$.
We also introduce some new notation, for any $\ell\in\Ncal_N$ we will use the notation $\frac 1N \ell$ to denote the sequence $\left(\frac {\ell_k}N\right)_{k\in\N}$, which clearly denotes an element of $\Ncal$. On the other hand, with $\ell_{\lfloor \cdot N\rfloor}$ we denote the point measure on $(0,1]$ with weight $\ell_{k}$ at the point $
\frac kN$ for any $k=1,\dots, N$ and zero everywhere else. This integer valued measure clearly belongs to $\Mcal$.

We start by looking at the term $F_{\rm Mi}(\ell)$,  i.e., the microscopic term and we combine it  with the first term in \eqref{StirlingAppl}.

\begin{lemma}\label{lem-upper-bound-F_Mi} Fix $\Lambda\in \Ncal$. Fix $\delta>0$ and pick sequences $\ell\in \Ncal_N$ and $R_N\to\infty$  such that $d_{R_N}(\frac 1N \ell, \Lambda)\leq \delta$ for all $N$.  Then, for any $R\in\N$, as $N\to\infty$,
\begin{equation}\label{Microesti}
\left(\frac{N}{\e}\right)^{N c_\mathrm{Mi}(\ell/N)}F_{\rm Mi}(\ell)
\leq \exp \left(-N
I_{\rm Mi}(\Lambda;t)\right)\,\e^{N (C_R(\delta)+\gamma_R)+o(N)}\,\e^{-N(\frac t2-\log t)(c_\mathrm{Mi}(\ell/N)-c_{\Lambda})},
\end{equation}
where $\lim\limits_{R\to\infty}\gamma_R=0$ and $\lim\limits_{\delta\downarrow0}C_R(\delta)=0$. 
\end{lemma}
\begin{proof}

For any fixed $k\leq R_N$, we use the upper bound in~\eqref{Upp_bound_k-1_edges}, the fact that $1-x\leq \e^{-x}$ and Stirling's lower bound for $\ell_k!$ (notice that for $k$ small we expect $\ell_k$ to be large, $\Uptheta(N)$) to obtain
$$
z_k(\ell)\leq \frac{k^{(k-2)\ell_k}t_N^{(k-1)\ell_k}\e^{-\frac {t_N}{2N} k(N-k)\ell_k}}{k!^{\ell_k}N^{(k-1)\ell_k} \, (\frac 1\e\ell_k)^{\ell_k}}.
$$

We obtain, uniformly for $\ell \in \Ncal_N$, using that $\sum_{k=1}^{R_N} \frac t{2N}k^2\ell_k\leq \frac{t}{2}R_N c_\mathrm{Mi}(\ell/N)$,
\begin{equation*}
 \begin{aligned}
\left(\frac{N}{\e}\right)^{N c_\mathrm{Mi}(\ell/N)}F_{\rm Mi}(\ell)&\leq\exp \Big(-N\sum_{k=1}^{R_N}\frac 1N\ell_k\log\frac{k! \e^k \frac 1N\ell_k}{k^{k-2}t^{k-1}\e^{1-\frac t2 k}}\Big)\,\e^{o(N)}\\
&=\exp \left(-N 
I_{\rm Mi}^{\ssup {R_N}}(\smfrac 1N\ell;t)\right)\,\e^{o(N)},
\end{aligned}
\end{equation*}
where 
\begin{equation}\label{eq:I_mi_trunc}
I_{\rm Mi}^{\ssup {R_N}}(\widetilde \Lambda;t)=f^{\ssup {R_N}}(\widetilde\Lambda;t)+\sum_{k=1}^{R_N}k\widetilde\lambda_k \Big(\frac t2-\log t \Big)\qquad\text{with } f^{\ssup {R_N}}(\widetilde\Lambda;t)\colon =\sum_{k=1}^{R_N} \widetilde \lambda_k
    \log\frac{k!t\e^{k-1}\widetilde\lambda_k}{ k^{k-2}},
\end{equation}
is the cut-off version of the rate function defined in \eqref{I_mi}. Recall that $d_{R_N}(\frac 1N\ell,\Lambda)<\delta$ and that $c_\Lambda=\sum_{k\in\N}k \lambda_k\in[0,1]$ and observe that $\lim_{R\to\infty}I^{\ssup{R}}_{\rm Mi}(\Lambda;t)=I_{\rm Mi}(\Lambda;t)$. 

To prove \eqref{Microesti}, we notice that $f^{\ssup {R}}(\cdot;t)$ is continuous, it is clear that $\sup\limits_{\ell\colon \d_R(\frac 1N\ell,\Lambda)< \delta}|f^{\ssup R}(\frac 1N\ell;t)-f^{\ssup R}(\Lambda;t)|$ vanishes as $\delta\downarrow 0$ and can therefore be estimated against such a $C_R(\delta)$. Moreover, we estimate (substituting $\frac 1N\ell$ by $\widetilde \Lambda$), for any $N$ such that $R_N>R$, with the help of the Stirling bound $k!\e^{k}k^{-k}\geq 1$ and  Jensen's inequality for $\varphi(x)=x\log x$, as follows:
\begin{equation}\label{jensen}
\begin{aligned}
f^{\ssup{R_N}}(\widetilde \Lambda;t)-f^{\ssup{R}}(\widetilde \Lambda;t)&=\sum^{R_N}_{k=R+1 } \widetilde \lambda_k 
    \log\frac{k!t \e^{k-1} \widetilde \lambda_k}{k^{k-2}}
    \geq 
\sum^{R_N}_{k=R+1 }\widetilde\lambda_k\log \frac{k^2t \widetilde \lambda_k }{\e}\\
&\geq \sum^{R_N}_{k=R+1}\frac {\e}{t k^2}\varphi\Big(\sum^{R_N}_{k=R+1 }\widetilde\lambda_k\Big/\sum^{R_N}_{k=R+1 }\frac {\e}{t k^2}\Big)\\
&=\sum^{R_N}_{k=R+1 } \widetilde\lambda_k \log \Big(\sum^{R_N}_{k=R+1 }\widetilde\lambda_k \Big/\sum^{R_N}_{k=R+1 }\frac {\e}{t k^2}\Big)\geq \sum^{R_N}_{k=R+1 } \widetilde\lambda_k \log \Big(cR\sum^{R_N}_{k=R+1 }\widetilde\lambda_k\Big)\\
&\geq -\gamma_R,
\end{aligned}
\end{equation}
for some $c>0$, where we used that the remainder sum $\sum_{k>R}\frac1{k^2}$ is of order $1/R$ as $R\to\infty$ and  that $\sum^{R_N}_{k=R+1 }\widetilde\lambda_k\leq 1/R$ since $\sum_k k\widetilde\lambda_k\leq 1$ and that the map $x\mapsto x\log(cRx)$ is decreasing in $(0,1/\e Rc)$, introducing some $-\gamma_R$ that vanishes as $R\to\infty$. This proves the claim \eqref{Microesti}. 
\end{proof}

Notice that the last term on the right-hand side of  \eqref{Microesti} cannot be further estimated with the help of continuity (since $\Lambda\mapsto c_\Lambda$ is not continuous), but will be jointly handled together with the correspondent macroscopic and mesoscopic terms. Next we focus on the term $F_{\rm Ma}(\ell)$,  the macroscopic term and we proceed analogously.

\begin{lemma}\label{lem-upper-bound_F_Ma} Fix $\alpha\in  \Mcal$. Fix $\rho>0$ and pick sequences $\ell\in \Ncal_N$ and  $\eps_N\downarrow 0$ such that $D_{\eps_N}(\ell_{\lfloor \cdot N\rfloor},\alpha)\leq \rho$ for all $N$. Further assume that $|\frac 1{\eps_N}\log\eps_N|\leq o(N)$.  Then, for any  $\epsilon>0$, as $N\to\infty$,
 \begin{multline}\label{Macroesti}
\left(\frac{N}{\e}\right)^{N c_\mathrm{Ma}(\ell/N)}F_{\rm Ma}(\ell)
 \leq \exp\big(-N I_{\rm Ma}(\alpha;t)\big)\,\e^{N (C_\eps(\rho)+\gamma_\eps+\frac t2 \eps)+o(N)}\, \e^{-N(\frac t2-\log t)(c_\mathrm{Ma}(\ell/N)-c_{\alpha})},
\end{multline}
for some $C_\eps(\rho)$ and $\gamma_\eps$ that satisfy $\lim_{\eps\downarrow 0}\gamma_\eps=0$ and $\lim_{\rho\downarrow0}C_\eps(\rho)=0$. 
\end{lemma}

\begin{proof}

We use the upper bound in Lemma~\ref{mu_asy_gen} and Stirling's lower bound for $k!$ (in this case we know that $k$ is large and we expect $\ell_k$ small). We obtain, for $k\in\{\eps_N N,\dots,N\}$,
$$
z_k(\ell)\leq \frac{(1-\e^{k q_N})^{(k-1)\ell_k}\e^{-\frac {t_N}{2N} k(N-k)\ell_k}}{k^{k\ell_k}\e^{-k\ell_k}\, \ell_k!},
$$
where $q_N=\log\left(1-\frac {t_N}N\right)$.
We pair $F_{\rm Ma}(\ell)$ with the second term in \eqref{StirlingAppl}, and we obtain, uniformly for $\ell \in \Ncal_N$, 
\begin{equation}\label{F_Ma(eps)}
 \begin{aligned}
\left(\frac{N}{\e}\right)^{N c_\mathrm{Ma}(\ell/N)}F_{\rm Ma}(\ell)
& \leq \prod_{\eps_N N\leq k\leq N} \Big[\Big(\frac{N}{\e}\Big)^{k\ell_k}
\frac{(1-\e^{k q_N})^{(k-1)\ell_k}\e^{-\frac {t_N}{2} k\ell_k}\e^{\frac {t_N}{2} k^2\ell_k/N}}{k^{k\ell_k}\e^{-k\ell_k}\, \ell_k!}\Big]\\  
&\leq \prod_{\eps_N N\leq k\leq N} \Big[\Big(\frac{N}{k}\Big)^{k\ell_k}
{(1-\e^{k q_N})^{(k-1)\ell_k}\e^{-\frac {t_N}{2} k\ell_k}\e^{\frac {t_N}{2} k^2\ell_k/N}}\Big]\\
&=\exp\big(-N I_{\rm Ma}^{\ssup {\eps_N}}(\ell_{\lfloor\cdot N\rfloor};t)+o(N)\big)\\
\end{aligned}
\end{equation}
where 
\begin{equation}\label{eq:I_Ma_trunc}
I_{\rm Ma}^{\ssup {\eps_N}}(\widetilde \alpha;t)=g^{\ssup {\eps_N}}(\widetilde \alpha;t)
+ \int_{[\eps_N,1]}x\,\widetilde \alpha(\d x)\Big(\frac t2-\log t \Big)
\end{equation}
with 
$$
g^{\ssup {\eps_N}}(\widetilde \alpha;t)= \int_{[\eps_N,1]} \Big[x\log\frac{tx}{1-\e^{-tx}}-\frac t2 x^2\Big]\,\widetilde \alpha(\d x),
$$
denotes the cut-off version of the rate function $I_{\rm Ma}$ defined in  \eqref{I_mi}. 
Indeed, for proving the last line of \eqref{F_Ma(eps)} we do the following. In the product, we add the factor $(1-\e^{-kt/N})^{k l_k}$ and its reciprocal, substitute $\exp\circ \log$ and turn the sum on $k$ into an integral over $x$. Then most of the terms are easily asymptotically identified with the corresponding terms in \eqref{eq:I_Ma_trunc}, with possible exception of the term
\begin{equation}\label{smallterm}
\prod_{\eps_N N\leq k\leq N}\Big[(1-\e^{k q_N})^{(k-1)\ell_k}(1-\e^{-kt/N})^{-k l_k}\Big],
\end{equation}
of which we now show that it is not larger than $\e^{o(N)}$. Now write $\ell$ in terms of $s=(s_i)_{i\in\{1,\dots,n\}}\in E_N$ defined in \eqref{state_space_N}, such that $\sum_i s_i=N$, and we pick $i^*$ minimal such that $s_{i*+1}<\eps_N N$. Then, also using the inequalities $\log(1+y)\leq y$ and $1-\e^{-x}\leq x$, we see that
$$
\begin{aligned}
\sum_{\eps_N N\leq k\leq N}&k\ell_k\log\frac {1-\e^{k q_N}}{1-\e^{-k\frac{t}N}}
=\sum_{i=1}^{i^*} s_i\log\frac {1-\e^{s_i q_N}}{1-\e^{-s_i t/N}} 
\leq \sum_{i=1}^{i^*} s_i \frac{\e^{-s_i t/N}-\e^{s_i q_N}}{1-\e^{-s_i t/N}}\\
&=\sum_{i=1}^{i^*} s_i \frac{\e^{-s_i t/N}}{1-\e^{-s_i t/N}}\Big(1-\e^{s_i(q_N+t/N)}\Big)
\leq -\sum_{i=1}^{i^*} s_i^2 \frac{\e^{-s_i t/N}}{1-\e^{-s_i t/N}}(q_N+t/N).
\end{aligned}
$$
Recall the definition of $q_N$ to see that $q_N+t/N\leq o(1/N)$.  Use Jensen's inequality and $\sum_i s_i\leq N$ to see that the entire last term is not larger than $o(N)$. To handle the last missing term in \eqref{smallterm}, notice (because of $\sum_k k \ell_k=N$) that $\sum_{k\geq \eps_N N}\ell_k\leq 1/\eps_N$ and $1-\e^{k q_N}\geq 1-\e^{\eps_N N q_N}\geq \eps_N N q_N\sim \eps_N t$ and therefore
\[
-\sum_{k\geq \lfloor \eps_N N\rfloor}\ell_k\log\left(1-\e^{k q_N}\right)\leq -\frac1{\eps_N}{\log (\eps_N t)}\leq o(N),
\]
where we recall that we assumed that $|\frac 1{\eps_N}\log\eps_N|\leq o(N)$. Hence, the term in \eqref{smallterm} is not larger than $\e^{o(N)}$.

To prove \eqref{Macroesti}, we first observe that $g^{\ssup\eps}(\cdot;t)$ is continuous and hence $|g^{\ssup\eps}(\ell_{\lfloor\cdot N\rfloor};t)-g^{\ssup\eps}(\alpha;t)|$ can be estimated against such a $C_{\eps}(\rho)$, uniformly in $N\in\N$ and $\ell$ such that $D_{\eps_N}(\ell_{\lfloor\cdot N\rfloor},\alpha)\leq\rho$. Furthermore, for any $\eps>0$ and any $N\in\N$ such that $\eps_N<\eps$, 
$$
g^{\ssup{\eps_N}}(\ell_{\lfloor\cdot N\rfloor};t)-g^{\ssup\eps}(\ell_{\lfloor\cdot N\rfloor};t)=\sum_{k=\eps_N N}^{\eps N } \ell_k\frac kN \left(\log \frac {\frac kN t}{1-\e^{-t\frac kN}}-\frac t2\frac kN\right)\geq -\frac t2 \eps,
$$
since $\log \frac{x}{1-\e^{-x}}\geq 0$ for all  $x>0$. Hence, we arrived at the bound in  \eqref{Macroesti}.
\end{proof}

Notice that again we refrain from estimating the term $\e^{-N(\frac t2-\log t)(c_\mathrm{Ma}(\ell/N)-c_{\alpha})}$, which needs to be coupled with the microscopic and the mesoscopic part. Then we are left to handle the middle term in~\eqref{ProbANsplit}.

\begin{lemma}\label{lem-upper-bound_F_Me} Fix $(\Lambda,\alpha)\in \Ncal\times \Mcal$ such that $c_\Lambda+c_\alpha\leq 1$. Fix $\delta,\rho>0$ and pick sequences $\ell\in \Ncal_N$ and $R_N\to\infty$ and $\eps_N\downarrow 0$ such that $d_{R_N}(\frac 1N \ell, \Lambda)\leq \delta$ and $D_{\eps_N}(\ell_{\lfloor \cdot N\rfloor},\alpha)\leq \rho$ for all $N$. Further assume that $|\frac 1{\eps_N}\log\eps_N|\leq o(N)$.  Then, as $N\to\infty$,
{\small \begin{equation}\label{Mesoesti}
\left(\frac{N}{\e}\right)^{N c_\mathrm{Me}(\ell/N)}F_{\rm Me}(\ell)\leq \big(t\e^{-t/2}\big)^{N c_\mathrm{Me}(\ell/N))+o(N)}=\exp\Big(-N\Big(\frac t2-\log t\Big)c_{\rm Me}(\ell/N)\Big)\,\e^{o(N)}.
\end{equation}}
\end{lemma}
\begin{proof}

We use again  the upper bound in~\eqref{Upp_bound_k-1_edges} and Stirling's formula, to see that 
\begin{align*}
\left(\frac{N}{\e}\right)^{N c_\mathrm{Me}(\ell/N)}F_{\rm Me}(\ell)&\leq \prod_{k=R_N+1}^{\lfloor\eps_N N\rfloor } \Big[\Big(\frac{N}{\e}\Big)^{k\ell_k}
\frac{k^{(k-2)\ell_k}(1-\e^{kq_N})^{(k-1)\ell_k }\e^{-\frac {t_N}{2N} k(N-k)\ell_k}}{k!^{\ell_k}\,\ell_k!}\Big]\\
&\leq \Big(\prod_{k=R_N+1}^{\lfloor\eps_N N\rfloor }\Big[\frac{N \e}{k^2 \ell_k t}\Big]^{\ell_k}\Big)\Big(\prod_{k=R_N+1}^{\lfloor\eps_N N\rfloor }\e^{\frac t{2N}k^2\ell_k}\Big)\,\big(t\e^{-t/2}\big)^{N c_\mathrm{Me}(\ell/N)}.
\end{align*}
We claim that the right-hand side is equal to $(t\e^{-t/2})^{N c_\mathrm{Me}(\ell/N)}\e^{N L_{N}(\ell)}$ for some $L_{N}(\ell)$ that vanishes, uniformly in $\ell$, as $N\to\infty$. First note that the next-to-last term is  such a term, since  $\frac {t}{2N}\sum_{k=R_N+1}^{\lfloor\eps_N N\rfloor }k^2\ell_k\leq \frac {t}{2}\eps_N N c_\mathrm{Me}(\ell/N)$. Furthermore, $\sum_{k=R_N+1}^{\lfloor\eps_N N\rfloor }\ell_k\leq N/R_N$, which shows that the terms containing $t$ and $\e$ in the first product are also so  small. 
With the same approach as in~\eqref{jensen}, we see the lower bound
$$
\liminf_{N\to\infty}\sum_{k=R_N+1}^{\lfloor\eps_N N\rfloor }\frac{\ell_k}N\log \frac{k^2 \ell_k }{N} \geq 0.
$$
Therefore, uniformly in $\ell$ such that $D_{\eps_N}(\ell_{\lceil\cdot N\rceil},\alpha)\leq\rho$, we have arrived at the estimate~\eqref{Mesoesti}.
\end{proof}
Now we collect the upper bounds above and substitute them in \eqref{ProbANsplit}, to obtain the following lemma.

\begin{lemma}\label{lem-upper-bound} Fix $(\Lambda,\alpha)\in \Ncal\times \Mcal$ such that $c_\Lambda+c_\alpha\leq 1$. Fix $\delta,\rho>0$ and pick sequences $\ell\in \Ncal_N$ and $R_N\to\infty$ and $\eps_N\downarrow 0$ such that $d_{R_N}(\frac 1N \ell, \Lambda)\leq \delta$ and $D_{\eps_N}(\ell_{\lfloor \cdot N\rfloor},\alpha)\leq \rho$ for all $N$. Further assume that $|\frac 1{\eps_N}\log\eps_N|\leq o(N)$.  Then, for any $R\in\N$ and $\epsilon>0$,
\begin{align*}
\limsup_{N\rightarrow \infty}\frac 1N \log
\P_{N,\frac 1N t_N}\big(A_{N}(\ell)\big)\leq  &-I(\Lambda,\alpha;t) + K_{R,\eps}(\delta,\rho),
\end{align*}
where $K_{R,\eps}(\delta,\rho)$ vanishes as $\delta\downarrow 0$ and $\rho\downarrow 0$, followed by $R\to\infty$ and $\eps\downarrow 0$.
\end{lemma}

\begin{proof}
We collect  the upper bound \eqref{Microesti} from Lemma~\ref{lem-upper-bound-F_Mi}, \eqref{Macroesti} from Lemma~\ref{lem-upper-bound_F_Ma} and~\eqref{Mesoesti} from Lemma~\ref{lem-upper-bound_F_Me}. We substitute them in \eqref{ProbANsplit}, also using \eqref{StirlingAppl}, then we obtain, uniformly in $\ell$ such that $d(\frac 1N\ell,\Lambda)<\delta$ and $D(\ell_{\lfloor\cdot N\rfloor},\alpha)<\rho$, for any $R\in\N$ and any $\eps>0$, as $N\to\infty$,
$$
\begin{aligned}
\frac 1N\log \P_{N,\frac 1N t_N}(A_{N,t}(\ell))
&\leq -I_{\rm Mi}(\Lambda;t)-I_{\rm Ma}(\alpha;t)+C_R(\delta)+\gamma_R+C_\eps(\rho)+\gamma_\eps+\frac t2 \eps\\
&\qquad - \Big(\frac t2-\log t\Big) (1-c_{\Lambda}-c_{\alpha})+o(1)\\
&=-I(\Lambda,\alpha;t) + K_{R,\eps}(\delta,\rho)+ o(1),
\end{aligned}
$$
where $K_{R,\eps}(\delta,\rho)$ vanishes as $\delta\downarrow 0$ and $\rho\downarrow 0$, followed by $R\to\infty$ and $\eps\downarrow 0$, and we recall that $c_{\rm Me}(\ell/N)= 1-c_{\rm Mi}(\ell/N)-c_{\rm Ma}(\ell/N)$. This implies the upper bound in \eqref{keyprob} in the case where $c_\Lambda+c_\alpha\leq 1$.

\end{proof}

In the following lemma, we implicitly use the lower semicontinuity of the maps $\Lambda\mapsto c_\Lambda$ and $\alpha\mapsto c_\alpha$ to show that when $c_{\Lambda}+c_{\alpha}>1$, then the event $A_{N,t}(\ell)$ is empty for any $\ell$ such that $d_{R_N}(\frac 1N\ell,\Lambda)\leq\delta$ and $D_{\eps_N}(\ell_{\lfloor\cdot N\rfloor},\alpha)\leq\rho$, if $\delta$ and $\rho$ are small enough. This will give the right super-exponential upper bound for $\P_{N,\frac 1N t_N}(A_{N,t}(\ell))$, since $I(\Lambda,\alpha;t)=\infty$.

\begin{lemma} Let $(\Lambda,\alpha)\in \Ncal\times \Mcal$ such that $c_\Lambda+c_\alpha> 1$, then there exists $R\in\N$, $\epsilon, \delta,\rho>0$  and $N_0$ large enough such that for all $N>N_0$: 
\begin{align*}
\left\{ \ell \in \Ncal_N\colon d_{R}(\smfrac 1N \ell, \Lambda)\leq \delta, \quad D_{\epsilon}(\ell_{\lfloor \cdot N\rfloor},\alpha)\leq \rho \right\}=\emptyset.
\end{align*}
\end{lemma}

\begin{proof}
We pick $R\in \N$ so large and $\eps\in(0,1)$ so small that $\sum_{k=1}^R k\lambda_k+\int_{[\eps,1]}x\,\alpha(\d x)$ are larger than one, say equal to $1+\eta $ for some $\eta>0$. Then choose $\delta$ and $\rho$ in $(0,1)$ so small that, for any $\ell$ such that $d_R(\frac 1N\ell,\Lambda)\leq\delta$  and $D_{\epsilon}(\ell_{\lfloor\cdot N\rfloor},\alpha)\leq\rho$, we have $\frac 1N\sum_{k=1}^Rk\ell_k
-\sum_{k=1}^R k\lambda_k\geq -\frac \eta 3$ and $ \frac 1N\sum_{k=R+1}^Nk\ell_k-\int_{[\eps,1]}x\,\alpha(\d x)\geq -\frac \eta 3$. Therefore we see that
\begin{align*}
1=\frac 1N\sum_{k=1}^Nk\ell_k &\geq \frac 1N\sum_{k=1}^Nk\ell_k
-\sum_{k=1}^R k\lambda_k-\int_{[\eps,1]}x\,\alpha(\d x)+\sum_{k=1}^R k\lambda_k+\int_{[\eps,1]}x\,\alpha(\d x)\\
&\geq -\frac \eta 3-\frac \eta3 +\sum_{k=1}^R k\lambda_k+\int_{[\eps,1]}x\,\alpha(\d x)> 1+\frac \eta3,
\end{align*}
which yields a contradiction.
\end{proof}

The remaining of this section deals with the construction of an optimal sequence $(\ell^{\ssup N})_{N\in\N}$, which will give a lower bound on the probability that matches the upper bound from Lemma~\ref{lem-upper-bound}. For $N$ large enough  define $\ell^{\ssup N} \in \Ncal_N$ by
\begin{equation}\label{recovery_seq}
\ell^{\ssup N}_k=\begin{cases}
\lfloor \lambda_k N\rfloor&\text{for }k=2,\dots,R_N-1;\\
\lfloor \frac{(1-c_{\lambda}-c_{\alpha}) N}{R_N}\rfloor&\text{for }k=R_N;\\
\alpha\big(\frac{k-1}{N},\frac kN\big] &\text{for }k= R_N+1, \dots,N;\\
N-\sum^N_{j\geq 2} j\ell^{\ssup N}_j & \text{for } k=1,
\end{cases}
\end{equation}
where $R_N$ is an arbitrary diverging sequence in $\N$ such that $R_N \ll N$.

We notice that our sequence $(\ell^{(N)})_{N\in\N}$ is such that the so-called mesoscopic mass is simply concentrated in components all of the same size (namely $R_N$) and surprisingly no specific requirement is imposed on $R_N$, except that it diverges. We will underline this in the steps of our proof.
It is clear by construction that the following hold
\begin{align}\label{conv_to_lambda}
\lim_{N\rightarrow \infty} d(\smfrac{1}{N}\ell^{\ssup N},\Lambda)&=0,\\ \label{conv_to_alpha}
\lim_{N\rightarrow \infty} D(\ell^{\ssup N}_{\lfloor\cdot N\rfloor},\alpha)&=0,\\
\end{align}
We now give lower bounds to $\P_{N,\frac 1N t_N}(A_{N,t}(\ell^{\ssup N}))$, starting from the formulation in \eqref{ProbANsplit} and \eqref{StirlingAppl}. By abuse of notation, we will drop the index $(N)$ from $\ell$.

\begin{lemma}\label{lem-lower-bound-F_Mi} Fix $(\Lambda,\alpha)\in \Ncal\times \Mcal$ such that  $I(\Lambda,\alpha;t)<\infty$ and let $\ell$ be defined by~\eqref{recovery_seq}. 
Then,  as $N\to\infty$,
\begin{equation}\label{eq_micro_lower}
 \begin{aligned}
\left(\frac{N}{\e}\right)^{N c_\mathrm{Mi}(\ell/N)}F_{\rm Mi}(\ell)&\geq\exp \left(-N 
I_{\rm Mi}^{\ssup {R_N}}(\Lambda;t)\right)\,\e^{o(N)},
\end{aligned}
\end{equation} 
where $I_{\rm Mi}^{\ssup {R_N}}(\Lambda;t)$ is defined in~\eqref{eq:I_mi_trunc}.
\end{lemma}
\begin{proof}

We use the lower bound in~\eqref{Upp_bound_k-1_edges} from Lemma~\ref{prob_finite_size} to perform a calculation similar to that for the upper bound (using now Stirling's upper bound on $\ell_k!$):
\begin{equation*}
 \begin{aligned}
\left(\frac{N}{\e}\right)^{N c_\mathrm{Mi}(\ell/N)}F_{\rm Mi}(\ell)&\geq\prod_{k=1}^{R_N} \Big[\Big(\frac{N}{\e}\Big)^{k\ell_k}
\frac{k^{(k-2)\ell_k}t_N^{(k-1)\ell_k }}{k!^{\ell_k}N^{(k-1)\ell_k} (\ell_k)^{\ell_k}\e^{-\ell_k}\sqrt{2\pi \ell_k}}\left(1-\frac{t_N}N\right)^{\frac 12 k\ell_k N-\frac 32 k\ell_k+\ell_k}\Big]\\
&=\exp \Big(-N\sum_{k=1}^{R_N}\frac 1N\ell_k\log\frac{k! \e^k \frac 1N\ell_k}{k^{k-2}t^{k-1}\e^{1-\frac t2 k}}\Big)\,\e^{o(N)}\\
&=\exp \left(-N 
I_{\rm Mi}^{\ssup {R_N}}(\Lambda;t)\right)\,\e^{o(N)},
\end{aligned}
\end{equation*} 
which proves the claim~\eqref{eq_micro_lower}.
\end{proof}
Notice that no restriction is required on $R_N$ in order for the above lower bound to coincide with the upper bound in the proof of Lemma~\ref{lem-upper-bound} (except that $R_N$ diverges). Indeed, although the lower bound in~\eqref{Upp_bound_k-1_edges} differs from the upper bound for a factor $\left(1-\frac {t_N}N\right)^{\frac{(k-1)(k-2)}2}$ the probability $\mu_{t_N}(k)$ is paired with $\left(1-\frac {t_N}N\right)^{\frac{k(N-k)}2}$ (the probability of a component to be separated from any other), which is an exact term and it balances the error coming from~\eqref{Upp_bound_k-1_edges}. In a similar way, one checks the following lower bound for the term involving $\ell_k$, for $k=R_N$.

\begin{lemma}\label{lem-lower-bound-F_Me} Fix $(\Lambda,\alpha)\in \Ncal\times \Mcal$ such that  $I(\Lambda,\alpha;t)<\infty$ and let $\ell$ be defined by~\eqref{recovery_seq}. 
Then,  as $N\to\infty$,
\begin{equation}
\label{eq_meso_lower}
\left(\frac{N}{\e}\right)^{N (1-c_{\lambda}-c_{\alpha})}z_{R_N}(\ell_{R_N})
\geq\exp\left(-(1-c_{\lambda}-c_{\alpha})\left(\frac t2-\log t\right)\right)\,\e^{o(N)}.
\end{equation}
\end{lemma}

\begin{proof}

\begin{multline*}
\left(\frac{N}{\e}\right)^{N (1-c_{\lambda}-c_{\alpha})}z_{R_N}(\ell_{R_N})\\
\geq
\Big(\frac{N}{\e}\Big)^{N (1-c_{\lambda}-c_{\alpha})}\frac{R_N^{(R_N-2)\ell_{R_N}}\left(\frac{t_N}N\right)^{(R_N-1)\ell_{R_N}}
\left(1-\frac {t_N}N\right)^{\frac 12 R_N\ell_{R_N} N-\frac 32 R_N\ell_{R_N}+\ell_{R_N}}}{R_N^{R_N\ell_{R_N}}\e^{-{R_N\ell_{R_N}}}(\sqrt{2\pi R_N})^{\ell_{R_N}}\ell_{R_N}^{\ell_{R_N}}\e^{-\ell_{R_N}}\sqrt{2\pi \ell_{R_N}}}\\
=\exp\left(-(1-c_{\lambda}-c_{\alpha})\left(\frac t2-\log t\right)\right)\,\e^{o(N)}.
\end{multline*}
The above lower bound relies on the lower bound~\eqref{Upp_bound_k-1_edges} in Lemma~\ref{prob_finite_size}, on Stirling's upper bound for $R_N!$ and $\ell_{R_N}!$ (which are both large) and on how we defined $\ell_{R_N}$ in~\eqref{recovery_seq}. 
\end{proof}
We notice that the term coming from Lemma~\ref{lem-lower-bound-F_Me}, in order to give the desired lower bound matching the upper bound from Lemma~\ref{lem-upper-bound}, does not add any condition on the sequence $R_N$, which is just supposed to diverge.

\begin{lemma}\label{lem-lower-bound-F_Ma} Fix $(\Lambda,\alpha)\in \Ncal\times \Mcal$ such that  $I(\Lambda,\alpha;t)<\infty$ and let $\ell$ be defined by~\eqref{recovery_seq}. 
Then, for all $\delta\in (0,1)$,  as $N\to\infty$, 
\begin{equation}\label{eq_macro_lower}
 \begin{aligned}
\left(\frac{N}{\e}\right)^{N c_\mathrm{Ma}(\ell/N)}F_{\rm Ma}(\ell)&\geq \exp \Big(-N\Big( I^{(\delta)}_{\rm Ma}(\alpha; t)+\big(\frac t2-\log t\big)\sum_{k=R_N+1}^{\delta N}\frac kN\alpha\big(\smfrac{k-1}{N},\smfrac kN\big]-r_{\delta}\Big)  +o(N)\Big),
\end{aligned}
\end{equation} 
where $I^{(\delta)}_{\rm Ma}(\alpha; t)$ is defined in \eqref{eq:I_Ma_trunc}, and  $r_\delta$ depends on $\int_0^\delta x\alpha(\d x)$ and it goes to zero when $\delta\searrow0$.
\end{lemma}
\begin{proof}

Now, fix $\delta\in (0,1)$. We use the lower bound \eqref{Upp_bound_k-1_edges} from Lemma \ref{prob_finite_size} and Stirling's upper bound on $k!$ and we write:
\begin{equation*}
\begin{aligned}
\prod_{k=R_N+1}^{\lfloor \delta N\rfloor} \Big(\frac{N}{\e}\Big)^{k\ell_k}
z_{k}(\ell_k)&\geq \prod_{k=R_N+1}^{\lfloor \delta N\rfloor} \Big[\Big(\frac{N}{\e}\Big)^{k\ell_k}
\frac{k^{(k-2)\ell_k}t_N^{(k-1)\ell_k }\left(1-\frac{t_N}N\right)^{\frac {kN}2\ell_k-\frac 32k\ell_k+
\frac {\ell_k}2}}{\ell_k!N^{(k-1)\ell_k} k^{k\ell_k}\e^{-{k\ell_k}}(2\pi k)^{\frac{\ell_k}{2}}}\Big]\\
&\geq \exp \Big(-N\big(\smfrac t2-\log t\big)\sum_{k=R_N+1}^{\delta N}\frac kN\alpha\big(\smfrac{k-1}{N},\smfrac kN\big] +\Rcal_{R_N,\delta}\Big),
\end{aligned}
\end{equation*}
where the remainder is defined as
\[
\Rcal_{R_N,\delta}=\sum_{k=R_N+1}^{\delta N}\ell_k\Big(\log N-\log t+\frac t{2N}-2\log k-\frac 12 \log (2\pi k)-\frac{3t}{2N}k\Big )-\sum_{k=R_N+1}^{\delta N}\log \ell_k!.
\] 
By defining $\ell_k:=\alpha\big(\frac{k-1}{N},\frac kN\big]$ and using that $\log k\leq k$, we see that
\[
\frac1{\delta}\int_{\frac{R_N+1}N}^\delta x\, \alpha(\d x)\leq \sum_{k=R_N+1}^{\delta N}\ell_k\leq \frac{N}{R_N}\int_{\frac{R_N+1}N}^\delta x\,\alpha(\d x),
\]
\[
 \sum_{k=R_N+1}^{\delta N}\ell_k\log k\leq N\int_{\frac{R_N+1}N}^\delta x\,\alpha(\d x).
\]
Finally, we use that $\sum_{h=1}^{\ell_k}\log h\leq \int_1^{\ell_k}\log (y\, )\d y$ and Jensen's inequality (since $x\log x$ is concave) to give the bound
\[
\sum_{k=R_N+1}^{\delta N}\log \ell_k!\leq \sum_{k=R_N+1}^{\delta N}\big(\ell_k\log \ell_k-\ell_k+1\big)\leq N\delta.
\] 
Exploiting the bounds above, we can say that 
\begin{equation}\label{lower_bound_small_macro}
\prod_{k=R_N+1}^{\lfloor \delta N\rfloor} \Big(\frac{N}{\e}\Big)^{k\ell_k}
z_{k}(\ell_k)\geq \exp \Big(-N\Big(\big(\frac t2-\log t\big)\sum_{k=R_N+1}^{\delta N}\frac kN\alpha\big(\smfrac{k-1}{N},\smfrac kN\big]-r_{\delta}\Big) +o(N)\Big),
\end{equation}
where $r_\delta$ depends on $\int_0^\delta x\alpha(\d x)$ and it goes to zero when $\delta\searrow0$. Notice that we could not handle directly the terms from $R_N+1$ to $N$ in one go, so we decided to fix a fictitious threshold $\delta N$, where $\delta$ can be chosen arbitrarily close to zero. This however does not affect the choice of $R_N$ and there is no need to add some condition on $R_N$ to handle such small, but macroscopic, components.

The remaining terms (the purely macroscopic ones) in \eqref{ProbANsplit} are treated as follows. We use Stirling's upper bound for $k!$ and we see that
\begin{equation*}
\begin{aligned}
\prod_{k=\lfloor \delta N\rfloor}^{N} \Big(\frac{N}{\e}\Big)^{k\ell_k}
z_{k}(\ell_k)&\geq \prod_{k=\lfloor \eps_N N\rfloor}^{\lfloor \delta N\rfloor} \Big[\Big(\frac{N}{\e}\Big)^{k\ell_k}
\frac{\mu_k(\frac 1N t_N)^{\ell_k}\left(1-\frac{t_N}N\right)^{\frac {kN}2\ell_k-\frac {k^2}2\ell_k}}{\ell_k!N^{k\ell_k} \left(\frac kN\right)^{k\ell_k}\e^{-{k\ell_k}}(2\pi k)^{\frac{\ell_k}{2}}}\Big]\\
&\geq \exp \Big(\sum_{\lfloor \delta N \rfloor+1}^N k\ell_k^\ssup{N}
  \Big[\log \left(\frac{\mu_k(\frac 1N t_N)^\frac1k}{k/N}\right) - \frac{t}{2N}(N-k)  \Big]\Big)\e^{o(N)}\\
  &=\exp \Big(N \int_{(\delta,1]} x \Big[\log\left(\frac{\mu_{\lfloor Nx\rfloor}(\frac 1N t_N)^{\frac{1}{\lfloor Nx\rfloor}}}{x}\right)- \frac{t}{2}(1-x)\Big]\alpha(\d x) \Big)\e^{o(N)}.
\end{aligned}
\end{equation*}
Let us focus on the integral $\int_{(\delta,1]} x \Big[\log\Big(\frac{\mu_{\lfloor Nx\rfloor}(\frac 1N t_N)^{\frac{1}{\lfloor Nx\rfloor}}}{x}\Big)- \frac{t}{2}(1-x)\Big]\alpha(\d x)$.
Now by \eqref{eq:macro_asymp} in Lemma~\ref{mu_asy_gen}, we know that the integrand converges pointwise to
\begin{equation*}
x \Big[\log\left(\frac{1-\e^{-tx}}{x}\right)- \frac{t}{2}(1-x)\Big].
\end{equation*}
Since, for $N$ large enough, 
\begin{equation*}
x \log\frac{x}{\mu_{\lfloor Nx\rfloor}(\frac 1N t_N)^{\frac{1}{\lfloor Nx\rfloor}}}
\leq x\left[\log\left(\frac{x}{1-\e^{-tx}}\right)+1\right],
\end{equation*}
which is clearly integrable over $x\in(\delta,1]$ with respect to $\alpha$, we can apply the dominated convergence theorem and we get 
\begin{equation}\label{lower_bound_macro}
\begin{aligned}
\prod_{k=\lfloor \delta N\rfloor}^{N} \Big(\frac{N}{\e}\Big)^{k\ell_k}
z_{k}(\ell_k)&\geq \exp \left(-N I^{(\delta)}_{\rm Ma}(\alpha; t)+o(N)\right),
\end{aligned}
\end{equation}
where $I^{(\delta)}_{\rm Ma}(\alpha; t)$ is defined in \eqref{eq:I_Ma_trunc}. Now, \eqref{lower_bound_small_macro} together with \eqref{lower_bound_macro} imply \eqref{eq_macro_lower}.
\end{proof}
Finally, we combine below the lower bounds in Lemma~\ref{lem-lower-bound-F_Mi}, \ref{lem-lower-bound-F_Me} and \ref{lem-lower-bound-F_Ma}.
\begin{lemma} \label{lem-lower-bound} For all $(\Lambda,\alpha)\in \Ncal\times \Mcal$ such that  $I(\Lambda,\alpha;t)<\infty$, there exists a sequence  $(\ell^{\ssup N})_{N\in\N}$ such that $\ell^{\ssup N }\in\Ncal_N$,  \eqref{conv_to_lambda} and \eqref{conv_to_alpha} hold, and
\begin{align}
\liminf_{N\rightarrow\infty} \frac 1N \log \P_{N,\frac 1N t_N}(A_{N,t}(\ell^{\ssup N}))&\geq-I(\Lambda,\alpha;t).\label{lower_bound}
\end{align}
\end{lemma}

\begin{proof}
From the lower bounds in \eqref{eq_micro_lower}, \eqref{eq_meso_lower} and \eqref{eq_macro_lower}, we get that there exists $a_\delta$ with $\lim_{\delta \searrow 0} a_\delta = 0$ and
\begin{equation}
\liminf_{N\rightarrow\infty} \frac 1N \log \P_{N,\frac 1N t_N}(A_{N,t}(\ell^{\ssup N}))\geq-I(\Lambda,\alpha;t) - a_\delta
\end{equation}
so \eqref{lower_bound} follows on taking the limit $\delta \searrow 0$.

\end{proof}

Now we are ready to prove Proposition \ref{Prop-mainLDP} by combining the above lemmas. 
\begin{proof}[Proof of Proposition \ref{Prop-mainLDP}]

Fix $\delta,\rho>0$ and $N\in\N$ and recall the definition of $A_{N,t}(\ell)$ in \eqref{statisticsevent}, then we see that
\begin{equation}\label{upp_bound_P}
\begin{aligned}
\P_{N,\frac 1N t_N}&\big({\rm Mi}^{\ssup N}\in B_\delta(\Lambda),\, {\rm Ma}^{\ssup N}\in B_\rho(\alpha)\big)\\
&=\sum_{\ell \in\Ncal_N}
\1\{d(\smfrac{1}{N}\ell,\Lambda)\leq \delta\}\,\1\{
D(\ell_{\lfloor\cdot N\rfloor},\alpha)\leq\rho\}\,\P_{N,\frac 1N t_N}(A_{N}(\ell)).
\end{aligned}
\end{equation}

Because of Lemma \ref{lem_card_N}, we only have to give asymptotic estimates on the single summands on the right-hand side of \eqref{upp_bound_P}. Let $(\Lambda,\alpha)\in \Ncal\times \Mcal$ such that $c_{\Lambda}+c_{\alpha}\leq 1$. We fix any diverging sequence $R_N$ and vanishing sequence $\eps_N$ with $R_N<\lfloor \eps_N N \rfloor$ and  $|\frac 1{\eps_N}\log\eps_N|\leq o(N)$, the we see that 
\[\begin{aligned}
\P_{N,\frac 1N t_N}&\big({\rm Mi}^{\ssup N}\in B_\delta(\Lambda),\, {\rm Ma}^{\ssup N}\in B_\rho(\alpha)\big)\\
&\leq \sum_{\ell \in\Ncal_N}
\1\{d_{R_N}(\smfrac{1}{N}\ell,\Lambda)\leq \delta\}\,\1\{
D_{\eps_N}(\ell_{\lfloor\cdot N\rfloor},\alpha)\leq\rho\}\,\P_{N,\frac 1N t_N}(A_{N}(\ell))\\
&\leq \exp\Big(-N I(\Lambda,\alpha; t)+N K_{R,\eps}(\delta,\rho)+ o(N)\Big),
\end{aligned}\]
as a consequence of Lemma \ref{lem-upper-bound}. Taking first $\delta$ and $\rho$ to zero and then $R\nearrow\infty$ and $\epsilon\searrow 0$, we get the desired upper bound. 

Given any fixed $\delta$ and $\rho$, we can construct the sequence $(\ell^{(N)})_{N\in\N}$ from Lemma \ref{lem-lower-bound} and see that, for such a sequence 
\[
\P_{N,\frac 1N t_N}\big({\rm Mi}^{\ssup N}\in B_\delta(\Lambda),\, {\rm Ma}^{\ssup N}\in B_\rho(\alpha)\big)\geq \P_{N,\frac 1N t_N}(A_{N}(\ell^{(N)}))\geq \exp\left(-N I(\Lambda,\alpha;t) - N a\right),
\]
for an $a$ arbitrarily small. 

Finally, if $(\Lambda,\alpha)\in\Ncal\times\Mcal$ are such that $c_{\Lambda}+c_{\alpha}>1$, Lemma \ref{lem-upper-bound} gives us that 
\[
\limsup_{N\to \infty}\frac 1N \log \P_{N,\frac 1N t_N}\big({\rm Mi}^{\ssup N}\in B_\delta(\Lambda),\, {\rm Ma}^{\ssup N}\in B_\rho(\alpha)\big)\leq -\infty.
\]
\end{proof}


\section{Corollaries and study of the rate functions}

\noindent In this section we analyse, for fixed $t\in[0,\infty)$,  the minima of the rate function,  $I(\Lambda,\alpha;t)$, over the configurations $\Lambda$ respectively $\alpha$, and afterwards the minimima of the rate functions for the total masses, $\Jcal_{\rm Mi}$, $\Jcal_{\rm Me}$ and $\Jcal_{\rm Ma}$. In particular, we will see in Lemma \ref{lem_min_lambda} that there is a drastic difference when minimizing $I(\Lambda,\alpha;t)$ over $\Lambda$ if $c_{\Lambda}\leq \frac 1t$ or not. Clearly, no such difference can be spotted when $t\leq 1$, since we do not allow for $c_{\Lambda}>1$. However, when $t>1$ we see that $c_{\Lambda}>\frac 1t$ is perfectly admissible and we interpret this as an analytic sign of the phase transition in $t=1$.

\subsection{Rate functions for the microscopic part}\label{sec-RFstat}

We start by minimizing $I(\Lambda,\alpha;t)$, for a fixed $\Lambda\in\Ncal$ over all compatible $\alpha\in\Mcal$. We will obtain the rate function for the microscopic part, and we will see that this minimum is attained for $\alpha$ of the form $\alpha=\delta_{c_{\alpha}}$.
Informally speaking, the following in particular implies that, with probability tending to one, there is at most one macroscopic particle.

\begin{lemma}[Analysis of the microscopic rate function]\label{lem_min_alpha}
Fix $\Lambda\in \Ncal$ and recall that $c_\Lambda=\sum_{k\in\N}k\lambda_k\in[0,1]$, then 
$$
\inf_{\alpha\in\Mcal_\N}I(\Lambda,\alpha;t)= I_{\rm Mi}(\Lambda;t)+I_{\rm Ma}(\delta_{1-c_{\Lambda}};t).
$$
where $I(\Lambda,\alpha;t)$, $I_{\rm Mi}(\Lambda;t)$ and $I_{\rm Ma}(\alpha;t)$ are defined in the statement of Theorem \ref{thm-LDP}.
\end{lemma}

\begin{proof}
Clearly 
$$
\begin{aligned}
\inf_{\alpha\in\Mcal_\N}I(\Lambda,\alpha;t)&=\inf_{c\in[0,1-c_\Lambda]}\inf_{\alpha\in\Mcal_{\N_0}(c)}I(\Lambda,\alpha;t)\\
&=I_{\rm Mi}(\Lambda;t)+\inf_{c\in[0,1-c_\Lambda]}\Big (\inf_{\alpha\in\Mcal_{\N_0}(c)}I_{\rm Ma}(\alpha;t)+(1-c_{\Lambda}-c)\Big(\frac t2-\log t\Big)\Big ).
\end{aligned}
$$
Fix $c\in[0,1]$ and $\alpha\in\Mcal_{\N_0}(c)$. Note that $\alpha((c,1])=0$ since $\alpha$ is a point measure with $\int_{(0,1]}x\,\alpha(\d x)=c$. We have, denoting   $f_t(x)= \log\frac{x}{1-\e^{-tx}}+\frac t2 (1-x)$,
\begin{equation}\label{min_tot_alpha}
I_{\rm Ma}(\alpha;t) =\int_{(0,c]} xf_t(x)\,\alpha(\d x)
\geq \int xf_t(c)\,\alpha(\d x)=cf_t(c)=I_{\rm Ma}(\delta_c;t),
\end{equation}
since  $f_t$ is strictly decreasing in $ [0,\infty)$. Indeed,
$$
f_t'(x)=\frac 1x-\frac{t \e^{-tx}}{1-\e^{-tx}}-\frac t2=\frac{t(1+y)}{2y(1-\e^{-2y})}\Big[\frac{1-y}{1+y}-\e^{-2y}\Big],\qquad y=\frac{tx}2.
$$ 
We want to prove that $f'(x)<0$ for $x\in [0,\infty)$. For $y\geq 1$, this is obvious from above, and for $y\in[0,1)$, this is easily seen as follows.
$$
\e^{2y}=1+\sum_{k=1}^{\infty}\frac{(2y)^k}{k!}<1+\sum_{k=1}^{\infty}2y^k=(1+y)\sum_{k=0}^{\infty}y^k=\frac{1+y}{1-y},
$$
since $\frac {2^k}{k!}<2$ for all $k\geq 3$.  Hence, we see that $f_t'(x)\leq 0$ for $x\in [0,\infty)$, and \eqref{min_tot_alpha} follows.

Furthermore, when we study 
$g_t(c):=c\log\frac{c}{1-\e^{-tc}}+\frac t2 c(1-c)+(1-c_{\Lambda}-c)(\frac t2-\log t)$, we see that its derivative is
\[
g'_t(c)=1- \frac{ct\e^{-ct}}{(1-\e^{-ct})}+\log \frac{ct\e^{-ct}}{(1-\e^{-ct})},\]
which is strictly negative if $\frac{ct\e^{-ct}}{(1-\e^{-ct})}\neq 1$. Since $\frac{ct\e^{-ct}}{(1-\e^{-ct})}<1$ if $ct>0$, we see that $g_t(c)$
is strictly decreasing in $c$, and hence the optimal value of $c$ is $c=1-c_{\Lambda}$.
\end{proof}

Now the proof of Corollary~\ref{cor-LDP_mi} directly follows from Theorem~\ref{thm-LDP}, Lemma~\ref{lem_min_alpha} and the contraction principle since the projection $(\Lambda,\alpha)\mapsto \Lambda$ is continuous in the product topology.
Let us mention that, since $c_{\Lambda}=\sum_{k\in\N}k\lambda_k$, we can rewrite 
\begin{equation}\label{eq_micro_entropy}
\Ical_{\rm Mi}(\Lambda;t)=\widehat I(\Lambda) - (1-c_{\Lambda})\left(\log\frac{1-\e^{(c_{\Lambda}-1)t}} {(1-c_{\Lambda})} - \frac{c_{\Lambda}t}{2}\right)+c_{\Lambda}\Big(\frac t2-\log t\Big)-
\frac 1{2t},
\end{equation}
with 
\begin{equation}\label{def:entropy}
\widehat I(\Lambda)= \sum_{k=1}^{\infty}(\lambda_k\log \frac {\lambda_k}{p_k}+p_k-\lambda_k), 
\end{equation}
with $p_k=\frac 1t\frac{k^{k-2}\e^{-k}}{k! }$ for all $k\in \N$.
Hence, the term $\widehat I(\Lambda)$ is  the relative entropy of two non-normalized measures $\Lambda$ and $p:=(p_k)_{k\in\N}$.
Notice that the reference measure $p$ is such that 
\[p_k=\frac{1}{tk}{\rm Bo}_1(k),\]
where 
\begin{equation}\label{eq:borel}
{\rm Bo}_{\mu}(k)=\frac{\e^{-\mu k}(\mu k)^{k-1}}{k!},\qquad k\in\N,
\end{equation}
are the probabilities of the Borel distribution with parameter $\mu\in[0,1]$.
The total mass of $p$ is therefore given by
\[
\sum_{k=1}^{\infty}p_k=\frac 1{t} \E\left[\frac1X\right],
\]
where $X$ is Borel distributed with parameter $1$. Now this expectation \cite[\S4.5]{AP98} is precisely $\frac12$, which explains why we added and subtracted the term $\frac 1{2t}$ to $\Ical_{\rm Mi}(\Lambda;t)$ in order to obtain the formulation \eqref{eq_micro_entropy}. As already mentioned in Section \ref{subsect:ERRG}, the above entropy form for the rate function \eqref{eq_micro_entropy} strictly relates to the rate function obtained in \cite[Theorem 1.8]{BorCap15} which also takes the form of an entropy with respect to a standard Galton Watson tree, whose total progeny is precisely Borel distributed. 

Let us analyse the minimising statistics of the macroscopic part.

\begin{lemma}[Analysis of the macroscopic rate function]\label{lem_min_lambda}
Fix $\alpha\in \Mcal_{\mathbb{N}_0}$ and recall that $c_\alpha=\int_{(0,1]}x\,\alpha(\d x)\in[0,1]$, then 
\begin{equation}\label{minlambda}
\inf_{\Lambda\in\Ncal}I(\Lambda,\alpha;t)= I_{\rm Ma}(\alpha;t)+C_{\alpha,t}\Big(\log (t C_{\alpha,t})-\frac t2 C_{\alpha,t}\Big)+(1-c_{\alpha}) \Big(\frac t2-\log t\Big),
\end{equation}
where $C_{\alpha,t}=(1-c_\alpha)\wedge \frac 1t$.

{Furthermore, the unique minimizer is equal to $\Lambda^*(C_{\alpha,t};t)$, defined in \eqref{lambdamin}.}
\end{lemma}

\begin{proof}
As in the proof of Lemma~\ref{lem_min_alpha}, we see that
\begin{equation}\label{minlambda1}
\begin{aligned}
\inf_{\Lambda\in\Ncal}I(\Lambda,\alpha;t)&=\inf_{c\in[0,1-c_\alpha]}\inf_{\Lambda\in\Ncal(c)}I(\Lambda,\alpha;t)\\
&=I_{\rm Ma}(\alpha;t)+\inf_{c\in[0,1-c_\alpha]}\Big (\inf_{\Lambda\in\Ncal(c)} I_{\rm Mi}(\Lambda;t)+(1-c_{\alpha}-c)\Big(\frac t2-\log t\Big)\Big )\\
&=I_{\rm Ma}(\alpha;t)+ (1-c_{\alpha})\Big(\frac t2-\log t \Big)-\frac 1{2t} +\inf_{c\in[0,1-c_\alpha]}\inf_{\Lambda\in\Ncal(c)}\widehat I(\Lambda),
\end{aligned}
\end{equation}
with $\widehat I(\Lambda)$ defined in \eqref{def:entropy}.
 Fix $c\in[0,1]$. Since $\widehat I$ is strictly  convex on the convex set $\Ncal(c)$, we see by evaluating the variational equations that the only candidate for a minimiser {in the interior} is
$$
\lambda^*_k(c;t)=\frac{\e^{k(\rho-1)} k^{k-2}}{k! t },\qquad k\in\N,
$$
with $\rho\in\mathbb{R}$ such that $\sum_{k=1}^{\infty}k\lambda^*_k(c;t)=c$. Interestingly, we can identify  $k\lambda^*_k(c;t)={\rm Bo}_{\mu_\rho}(k)\mu_\rho/t$, where $\mu_\rho$ is determined by $\mu_\rho-\log\mu_\rho=1-\rho$ and ${\rm Bo}_{\mu}$ is defined in \eqref{eq:borel}. Note that ${\rm Bo}_\mu(k)$ is not summable for $\mu>1$. Hence, $\rho$ must be picked such that $c=\mu_\rho/t$. The largest value $c$ that can be realised in this way is $c=1/t$ by picking $\rho=0$.  Hence, the preceding is possible at most for $c\in [0,1\wedge \frac 1t]$. By continuity and strict monotonicity of $\sum_{k=1}^{\infty}k\lambda^*_k(c;t)$ in $\rho$, indeed, any $c\in [0,1\wedge \frac 1t]$ can be uniquely realized, by picking $\rho=-tc +\log tc+1\leq 0$ such that $\sum_{k=1}^{\infty}k\lambda^*_k(c;t)=c$.  In this case, it is clear that the minimizer of $\widehat I$ in the interior of $\Ncal(c)$ is equal to
$$
\lambda^*_k(c;t)=\frac{k^{k-2}c^k t^{k-1}\e^{-ctk}}{k!},\qquad k\in\N,
$$
as claimed in~\eqref{lambdamin}, with value
{
\begin{equation}\label{lambdacmin}
\widehat I(\Lambda^*(c;t))= c(\log ct+1-ct)-\sum_{k=1}^{\infty}\lambda^*_k(c;t)+
\frac 1{2t}=c\Big(\log tc-\frac{tc}2\Big)+
\frac 1{2t},
\end{equation} 
where we used that if $X\sim {\rm Bo}_{ct}$, then $ \sum_{k=1}^{\infty}\lambda^*_k(c;t)=\mathbb{E}\left[\frac{1}{X}\right]=1-\frac{ct}{2}$, see \cite[\S4.5]{AP98}.}
Now we give an argument why $\Lambda^*(c;t)$ realises the minimum of $\widehat I$ over $\Ncal(c)$. We show that any such minimiser must be positive in every component. Indeed, if  ${\lambda}_{k^*}=0$ for some $k^*\in\N$, then we consider $\widehat\Lambda\in\Ncal(c)$, defined by
$$
\widehat\lambda_k=\begin{cases}
                   \eps,&\text{if }k=k^*,\\
                 {  \lambda_{\widehat k}-\eps C,}&\text{if }k=\widehat k,\\
                   \lambda_k&\text{otherwise,}
                  \end{cases}
$$
with $\widehat k\in\N\setminus\{k^*\}$ such that $\lambda_{\widehat k}>0$ and $C>0$ such that $\widehat\Lambda\in\Ncal(c)$ for any sufficiently small $\eps>0$. Now a simple insertion shows that $\widehat I(\widehat \Lambda)<\widehat I(\Lambda)$, if $\eps>0$ is small enough, since the slope of $\eps\mapsto \eps\log \eps$  at zero is $-\infty$. Hence, $\Lambda$ cannot be a minimizer. On the other hand, $\Lambda^*(c;t)$ has the property that all directional derivatives of $\widehat I$ in all admissible directions with compact support are zero; hence it is the minimizer of $\widehat I$ over $\Ncal (c)$ for $c\in[0,\frac 1t]$.

    
When $c>\frac 1t$, it is possible to pick a sequence of $\Lambda^{\ssup n}\in\Ncal(c)$ such that $\lim_{n\to\infty}\widehat I(\Lambda^{\ssup n})=0$ (pick $\lambda_k^{\ssup n}$ as $\lambda^*_k(\frac 1t;t)+\eps_n \delta_n(k)$ for some suitable $\eps_n>0$). 
Furthermore, since $\widehat I(\Lambda)$ is a relative entropy, we know that $\inf_{\Lambda\in\Ncal}\widehat I(\Lambda)\geq 0.$ 
Hence, the infimum of $\widehat I$ over $\Lambda\in\Ncal(c)$ for $c\geq \frac 1t$ is equal to $0$. This shows that the infimum over $\Lambda\in\Ncal(c)$ in the last line of \eqref{minlambda1} is equal to
$(c\wedge \frac 1t)(\log\left(t(c\wedge \frac 1t)\right)-\frac t2 (c\wedge \frac 1t))$, and \eqref{minlambda} follows.
\end{proof}

Then, the proof of Corollary~\ref{cor-LDP_ma} directly follows from Theorem~\ref{thm-LDP}, Lemma~\ref{lem_min_lambda} and the contraction principle, since the projection is continuous.\\

Finally, let us draw some conclusions regarding the mesoscopic mass. As stated after Corollary~\ref{cor-LDP_me}, it is not possible to apply the contraction principle, if we want to derive an LDP for the sequence of random variables $\overline{\rm Me}^{\ssup N}_{R_N,\eps_N}(t)$, however we can still identify the rate function by minimizing $I$ over all pairs $(\Lambda,\alpha)$ such that $c_{\Lambda}+c_{\alpha}=1-c$. Even if the contraction principle cannot be applied directly, the following lemma proves that the rate function $\Jcal_{\rm Me}(c;t)$ has exactly the expected form, given by~\eqref{minlambda_alpha}. 

\begin{lemma}\label{Lem-mesoLDP}
Fix $t\in[0,\infty)$. Then, for any $c\in[0,1]$ and any $R_N\in\N$ and $\eps_N\in(0,1)$ such that $1\ll R_N< \eps_N N\ll N$, 
\begin{equation}\label{Meso_prob}
\lim_{\delta\downarrow 0}\lim_{N\rightarrow \infty}\frac 1N \log
\P_{N,\frac 1N t_N}\Big(\big|\overline{\rm Me}^{\ssup N}_{R_N,\eps_N}(t)-c\big|\leq \delta\Big)=-\Jcal_{\rm Me}(c;t).
\end{equation}
\end{lemma}

\begin{proof} We first verify that, for a 
fixed $c\in[0,1]$, 
\begin{equation}\label{minlambda_alpha}
\inf_{\substack{\Lambda\in\Ncal\\\alpha\in \Mcal_{\mathbb{N}_0}\\ c_{\Lambda}+c_{\alpha}=1-c}}I(\Lambda,\alpha;t)= \Jcal_{\rm Me}(c;t)=(1-c)\Big ( \log(1-c)t-\frac{(1-c)t}{2}\Big)+\frac t2-\log t\Big).
\end{equation}
 Fix $x\in [0,1-c]$, then for a fixed $\Lambda\in \Ncal(x)$
\begin{multline*}
\inf_{\alpha\in \Mcal_{\mathbb{N}_0}(1-c-x)}I(\Lambda,\alpha;t)
=I_{\rm Mi}(\Lambda;t) \\
+ c\Big(\frac t2-\log t\Big)+\Big[(1-c-x)\log\frac{1-c-x}{1-\e^{-t(1-c-x)}}+\frac t2 (1-c-x)(c+x)\Big],
\end{multline*} 
since the infimum is attained in $\alpha=\delta_{1-c-x}$, as proved in Lemma~\ref{lem_min_alpha}. Then, with the same procedure of Lemma~\ref{lem_min_lambda}, we see that the infimum over $\Lambda\in\Ncal(x)$ is attained in $$\lambda_k^*(x\wedge \smfrac 1t;t)=\left(x\wedge \smfrac 1t\right)\frac{\e^{-\left(x\wedge \frac 1t\right)tk}\left(\left(x\wedge \frac 1t\right)t\right)^{k-1}k^{k-2}}{k!},
$$
giving 
\begin{equation}\label{opt_meso}
\begin{aligned}
\inf_{\substack{\alpha\in \Mcal_{\mathbb{N}_0}(1-c-x)\\
\Lambda\in\Ncal(x)}}I(\Lambda,\alpha;t)
&= \left(x\log tx-\frac {tx}2+\frac 1{2t}\right)\mathds{1}(x<\smfrac 1t)+(c+x)\left(\frac t2-\log t\right)-\frac 1{2t}\\
&\qquad +(1-c-x)\log\frac{1-c-x}{1-\e^{-t(1-c-x)}}+\frac t2 (1-c-x)(c+x).
\end{aligned}
\end{equation}
Minimizing then for $x\in[0,1-c]$, we see that the infimum is attained in $x^*$, the smallest solution to 
$$
x^*=(1-c)\e^{-t(1-c-x^*)},
$$
which is $x^*=1-c$, for all $t\geq\frac{1}{1-c}$ and $x^*< 1-c$ otherwise. By substituting the optimal $x^*$ in \eqref{opt_meso}, we see that~\eqref{minlambda_alpha} holds.\\

Now, notice that the procedure to get the upper bound in the proof of Proposition~\ref{Prop-mainLDP}  implies in a straightforward way that 
$$
\lim_{\delta\downarrow 0}\limsup_{N\rightarrow \infty}\P_{N,\frac 1N t_N}\Big(\big|\overline{\rm Me}^{\ssup N}_{R_N,\eps_N}(t)-c\big|\leq \delta\Big)\leq-\inf_{\substack{\Lambda\in\Ncal\\\alpha\in \Mcal_{\mathbb{N}_0}\\ c_{\Lambda}+c_{\alpha}=1-c}}I(\Lambda,\alpha;t)=-\Jcal_{\rm Me}(c;t).
$$
In the same way, from the proof of Proposition~\ref{Prop-mainLDP}, we borrow the strategy of constructing a ``recovery sequence'', this time using $\Lambda^*(x^*;t)$ and $\alpha^*=\delta_{1-c-x^*}$ to construct $\ell^{\ssup N}$ as in~\eqref{recovery_seq}. This gives $$
\lim_{\delta\downarrow 0}\liminf_{N\rightarrow \infty}\P_{N,\frac 1N t_N}\Big(\big|\overline{\rm Me}^{\ssup N}_{R_N,\eps_N}(t)-c\big|\leq \delta\Big)\geq-\Jcal_{\rm Me}(c;t).
$$
\end{proof}

The proof of the second point in Corollary~\ref{cor-LDP_me} follows as a direct consequence of Lemma~\ref{Lem-mesoLDP}.

\subsection{Proof of Theorem~\ref{thm-phasetrans}}\label{sec-Proofphasetrans}

Item (1) follows by Lemma~\ref{lem_min_alpha} and~\ref{lem_min_lambda}. Following the approach of those proofs, one can easily see that the order of minimization is not important, in particular:
$$
\Jcal_{\rm Mi}(c;t)=\inf_{\Lambda \in \Ncal(c)} \Ical_{\rm Mi}(\Lambda;t)=\inf_{\alpha \in \Mcal_{\mathbb{N}_0}(1-c)} \Ical_{\rm Ma}(\alpha;t)=\Jcal_{\rm Ma}(1-c;t).
$$
The minimizer, given a certain microscopic mass $c\in[0,1]$, is seen to take the form 
\[
(\Lambda^*(c\wedge \smfrac 1t; t),\delta_{1-c}),
\]
where $\Lambda^*$ is defined in \eqref{lambdamin} and $\alpha=\delta_{1-c}$ is a single macroscopic component. Imposing a certain microscopic (respectively macroscopic) mass influences the optimal configuration. Indeed, although it is optimal for the system to avoid mesoscopic mass (as seen in Corollary \ref{cor-LDP_me}(2)), the impossibility of minimizing the microscopic configuration under a certain constraint on the mass, namely $c_{\Lambda}>\frac 1t$, forces the system to actually have a mesoscopic mass of size $c_{\Lambda}-\frac 1t$. The same happens when we impose a macroscopic mass which is too small, namely $c_{\alpha}<\frac {t-1}t$. 
The form of the function $\Jcal_{\rm Mi}(c;t)$ in \eqref{min_lambda} comes directly from such minimization procedures.

Let us now prove assertion (2). The form of the minimizing $\Lambda$ follows from Lemma~\ref{lem_min_lambda}. Fix $t\in[0,1]$. Then $\Jcal_{\rm Mi}(c;t)=c\log c- t c^2+t c +(1-c)\log \frac{1-c}{1-\e^{t(c-1)}}$ is strictly decreasing in $c\in[0,1]$. Indeed 
$$
\frac {\d}{\d c}\Jcal_{\rm Mi}(c;t)=\log tc -tc+\frac{t(1-c)\e^{-t(1-c)}}{1-\e^{-t(1-c)}}-\log \frac{t(1-c)\e^{-t(1-c)}}{1-\e^{-t(1-c)}}=F\Big(\frac{t(1-c)\e^{-t(1-c)}}{1-\e^{-t(1-c)}}\Big)-F(tc),
$$
where we introduced the function $F(x)=x-\log x$, which is decreasing in $x\in(0,1]$. Hence, monotonicity of $\Jcal_{\rm Mi}(\cdot;t)$  in $[0,1]$ follows from
\begin{equation}\label{inequalities}
tc\leq \frac{t(1-c)\e^{-t(1-c)}}{1-\e^{-t(1-c)}}\leq1.
\end{equation}
The first inequality follows by observing that the function $\phi_t(c)=\e^{-t(1-c)}-c$ is nonnegative for all $c\in[0,1\wedge\frac 1t]$, since  $\phi_t(0)=\e^{-t}> 0$, $\phi_t(1)=0$, and $\phi_t$ is strictly decreasing in $[0,1]$, since $t\leq 1$. The second inequality follows from the fact that $\psi(z):=1-\e^{-z}-z\e^{-z}\geq0$ for all $z\in [0,1]$ (substitute $z=t(1-c)$), since $\psi(0)=0$, $\psi(1)=1-2\e^{-1}\geq0$ and $\psi$ is strictly increasing in $[0,1]$. Therefore, $\Jcal_{\rm Mi}(\cdot;t)$ is minimized in $c=1$, which implies the conclusion.

Now we turn to assertion (3). For $t\in(1,\infty)$, the derivative of $\Jcal_{\rm Mi}(c;t)$ writes as follows
 \begin{equation*}
\begin{aligned}
 \frac {\d}{\d c}\Jcal_{\rm Mi}(c;t)&=\frac{t(1-c)\e^{-t(1-c)}}{1-\e^{-t(1-c)}}-\log \frac{t(1-c)\e^{-t(1-c)}}{1-\e^{-t(1-c)}}+
\begin{cases}
\log tc -tc&\text{for }c\leq \frac 1t,\\
-1 &\text{for }c>\frac 1t.
\end{cases}
\end{aligned}
\end{equation*} 
It is clear that  $\Jcal_{\rm Mi}(c;t)$ is strictly increasing in $c\in(\frac 1t,1]$, while for $c\in[0,\frac 1t]$, we need to go back to \eqref{inequalities}. The right inequality there is still true for any $c<\frac 1t$. Since the quotient in \eqref{inequalities} is strictly increasing in $c$ and since $F(x)=x-\lg x$ is strictly convex in $x$, the unique zero of $\frac{\d}{\d c}\Jcal_{\rm Mi}(c;t)$ is given by the unique solution $c$ of
$$
tc=\frac{t(1-c)\e^{-t(1-c)}}{1-\e^{-t(1-c)}},$$
which is precisely the solution $c=\beta_t$ of~\eqref{equilibrium}. The remaining assertions follow.

\section*{Acknowledgements}
The authors acknowledge three anonymous referees for their careful reviews and many suggestions for improving the exposition. 
This research has been funded by the
Deutsche Forschungsgemeinschaft (DFG) through grant 
CRC 1114 ``Scaling Cascades in Complex Systems'', 
Project C08.

\bibliographystyle{alpha}
\bibliography{article}

\end{document}